\documentclass[12pt]{article}
\usepackage{graphicx}
\usepackage{subcaption}
\usepackage[utf8]{inputenc}
\usepackage{xcolor}
\usepackage[a4paper, total={6in, 8in}]{geometry}
\usepackage{amsfonts, amsmath, amssymb, amsthm}
\usepackage{hyperref}
\usepackage{url}
\usepackage{csquotes}
\usepackage[style=apa]{biblatex}

\newtheorem{Th}{Theorem}

\newtheorem{Cor}[Th]{Corollary}
\newtheorem{Lema}[Th]{Lemma}

\theoremstyle{definition}

\newtheorem{Def}[Th]{Definition}

\theoremstyle{remark}

\newcommand{\rfrac}[2]{^{#1}\!/_{#2}} 

\makeatletter
\newcommand{\labitem}[2]{%
	\def\@itemlabel{\textbf{#1}}
	\item
	\def\@currentlabel{#1}\label{#2}}
\makeatother
\newcommand{\footremember}[2]{%
    \footnote{#2}
    \newcounter{#1}
    \setcounter{#1}{\value{footnote}}%
}
\newcommand{\footrecall}[1]{%
    \footnotemark[\value{#1}]%
} 
\providecommand{\keywords}[1]{\textbf{\textit{Keywords}} #1}
\providecommand{\MSC}[1]{\textbf{\textit{MSC 2020 subject classification}} #1}

\title{Distributional Convergence of Empirical Entropic Optimal Transport and Statistical Applications}

\author{Santiago Arenas-Velilla \footremember{ims}{\scriptsize Institute for Mathematical Stochastics, University of G\"ottingen, Goldschmidtstra{\ss}e 7, 37077 G\"ottingen} 
			\and 
			Axel Munk \footrecall{ims} \footnote{\scriptsize Cluster of Excellence ``Multiscale Bioimaging: from Molecular Machines to Networks of
Excitable Cells" (MBExC),University Medical Center, Robert-Koch-Stra{\ss}e 40, 37075 G\"ottingen, Germany}
			\and
			Luis-Alberto Rodr\'iguez \footrecall{ims}{}
						}
\date{\today}

\begin{document}

\maketitle

\begin{abstract}
	Recently, the statistical properties of empirical Entropic Optimal Transport (EOT) have attracted great interest, as this quantity has been shown to be useful for complex data analysis, among other reasons due to its computational efficiency. In several applications, it has been observed that the EOT plan provides valuable information beyond just the optimal value. For example, in cell biology, colocalization analysis based on the EOT plan has been introduced as a measure for quantification of spatial proximity of different protein assemblies. Despite recent progress in the analysis of its risk properties, a precise understanding of its statistical fluctuations to make it accessible for inference remains elusive to a large extent. In this paper, we derive asymptotic weak convergence result for a large class of functionals of the EOT plan, in which the colocalization process is included. The proof is based on Hadamard differentiability and the extended delta method. As an application, we obtain uniform confidence bands for colocalization curves and bootstrap consistency. Our theory is supported by simulation studies and is illustrated by real world data analysis from mitochondrial protein colocalization.
\end{abstract}

\keywords{Bootstrap; colocalization analysis; entropic optimal transport; Hadamard directional differentiability; limit laws; uniform confidence bands.}

\MSC{Primary: 62G20, 62E20 Secondary: 62-08.}

\section{Introduction}\label{Sec:Intro}
		\subsubsection*{From optimal transport to entropic optimal transport}
			Since the pioneering work of \cite{Monge1781}, optimal transport (OT) has become the subject of significant interest in many areas, with a long standing history in mathematics and economics, as evidenced by seminal works such as \cite{Kantorovich1942}, \cite{Rachev&Rueschendorf1998}, \cite{Villani2008}, \cite{Santambrogio2015} and \cite{Galichon2018}. Besides its theoretical appeal and its immediate use for solving transportation and related optimization problems (see e.g. \cite{Vanderbei2020}), more recently the OT distance and related quantities, such as the OT map, have been recognized as valuable tools for data analysis. Such methodology has been applied in different areas: in computer science (\cite{Schmitz2018} and \cite{Balikas2018}), mathematical imaging (\cite{Ferradans2014} and \cite{Adler2017}), machine learning (\cite{Arjovsky2017} and \cite{Sommerfeld2019}) and computational biology (\cite{Schiebinger2019} and \cite{Tameling2021}) to mention a few.
			
			An important contribution towards the practical use of OT based data analysis has been provided by computational advances. The so called \enquote{entropic relaxation} of the original optimization problem has become particularly prominent (\cite{Cuturi2013}), as it allows for an iterative diagonal matrix multiplication algorithm which can improve the run time of exact solvers by an order of magnitude (see \cite{Cuturi&Peyre2019}). This, and further computational improvements (e.g \cite{Altschuler2017}, \cite{Dvurechensky2018} or \cite{Kassraie2024}), paved the way to analyse large scale data sets based on entropic optimal transport (EOT) in a routine daily data analysis on a laptop (see, e.g., the POT package \cite{flamary2021pot} for a \verb!Python! implementation, the \verb!R!-package \emph{transport}, \cite{Schrieber2016}, and the \verb!Julia! package \emph{MuSink}, \cite{Musink2025Staudt}).
			
			Consequently, there also emerged great interest recently in the statistical properties of EOT. We mention \cite{Nutz2021}, for a proof of existence and uniqueness of solutions for the EOT problem in large generality (see also \cite{Carlier&Laborde2020} for an alternative approach based on differentiability of the corresponding functional). Relevant to a valid statistical justification of the use of EOT are risk bounds and the investigation of the influence of the regularization parameter, see e.g. \cite{Genevay2016}, \cite{Genevay2019}, \cite{Groppe&Hundrieser2024}, \cite{Pal2024} and \cite{Sadhu2025} (and references therein). A different branch of literature is concerned with central limit theorems for the EOT-value and related quantities initiated by \cite{Klatt2020}, see \cite{Gonzalez2022}, \cite{Gonzalez&Hundrieser2023} and \cite{Goldfeld2024}. Meanwhile, the literature on the topic is vast and we refer to \cite{Balakrishnan2025}, \cite{delBarrio2025} and \cite{Chewi2025} for review of the recent contributions to the statistical analysis of EOT.
			
			However, a rigorous a treatment of the regularized OT \emph{plan} as a stochastic process is still missing, albeit this could help to provide a concise statistical theory for certain functionals of the plan which have been proposed recently. The goal of this work is to close this gap. For example, in \cite{Klatt2020}, so called EOT colocalization curves have been introduced to describe the relative amount of mass on a given spatial scale to match two proteins assemblies (see also \cite{naas2024multimatch} for an extension to unbalanced EOT and \cite{Vaisey_2022}, \cite{Strong_2023} and \cite{Hirtl_2025} for applications in different areas of cell biology). Among others, our results provide asymptotic uniform confidence \emph{bands} for EOT-colocalization curves based on its limit process. We stress that previous approaches only could provide pointwise confidence \emph{intervals}, see e.g. \cite{Klatt2020}, \cite{Gonzalez2022} or \cite{Liu2025}.
			
			To be more precise, we introduce EOT as the solution of the following optimization problem: Let $\mathcal{X}$ be a Polish space and let $\mu$, $\nu$ two probability measures on $\mathcal{X}$ and $c:\mathcal{X}\times\mathcal{X}\to\mathbb{R}$ be a measurable non-negative cost function. The \emph{entropic optimal transport plan} $\pi_{\mu,\nu}^{\lambda}$ for a regularization parameter $\lambda>0$ is defined as the unique minimizer of the (strictly convex) \emph{entropically regularized optimal transport functional}
			\begin{equation}\label{Eqn:EOT}
				\operatorname{EOT}_{c}^{\lambda}(\mu,\nu)=\min_{\pi\in\Pi(\mu,\nu)}\,\left(\int_{\mathcal{X}^{2}}c(x,y)\,\operatorname{d}\!\pi(x,y)+\lambda\,\operatorname{KL}(\pi|\mu\otimes\nu)\right),
			\end{equation}
			over all couplings between the marginals $\mu$ and $\nu$,
			\begin{equation*}
				\begin{aligned}
					\Pi(\mu,\nu)=&\left\{\pi\in\mathcal{P}\left(\mathcal{X}^{2}\right):\pi(A\times\mathcal{X})=\mu(A),\right.
					\\
					&\left.\pi(\mathcal{X}\times A)=\nu(A)\quad \operatorname{for}\ A\in\mathcal{B}_{\mathcal{X}}\right\}.
				\end{aligned}
			\end{equation*}
			Here, $\mathcal{P}\left(\mathcal{X}^{2}\right)$ is the set of Borel probability measures on $\mathcal{X}^{2}$, $\mathcal{B}_{\mathcal{X}}$ is the Borel $\sigma$-field of $\mathcal{X}$, and
			\begin{equation*}
				\operatorname{KL}(\pi|\mu\otimes\nu)=\left\{\begin{array}{ll}
					\int_{\mathcal{X}^{2}}\log\left(\frac{\operatorname{d}\!\pi}{\operatorname{d}\!\mu\otimes\nu}\right)\,\operatorname{d}\!\pi&\pi\ll\mu\otimes\nu
					\\
					\infty&\operatorname{otherwise}
				\end{array}\right.,
			\end{equation*}
			where $\ll$ denotes \enquote{absolute continuity of measures}. Existence and uniqueness of $\pi_{\mu,\nu}^{\lambda}$ is guaranteed by \cite[Theorem 4.2]{Nutz2021}. $\operatorname{KL}(\pi|\mu\otimes\nu)$ is known as the Kullback-Leibler divergence between $\pi$ and $\mu\otimes\nu$ and plays the role of a penalization term which assigns more weight to higher cost transportation. For the use of alternative penalization terms we refer to \cite{Klatt2020}, \cite{Chapel2021} and references therein. 
			It is worth to mention that if $\lambda=0$ in \eqref{Eqn:EOT}, the (non penalized) optimal transport problem is recovered, see, e.g., \cite{Altschuler2022}, \cite{Nutz&Wiesel2022}, \cite{Groppe&Hundrieser2024}, \cite{Manole2024} or \cite{Rigollet&Stromme2025} among others. It is well known that \eqref{Eqn:EOT} also admits the dual formulation (\cite[Theorem 3.2]{Nutz2021})
			\begin{equation}\label{Eqn:DualEOT}
				\begin{aligned}
					\operatorname{EOT}_{c}^{\lambda}(\mu,\nu)&=\sup_{\substack{
							f\in\operatorname{L}_{1}(\mu)
							\\
							g\in\operatorname{L}_{1}(\nu)
					}}\,\left(\int_{\mathcal{X}}\left(f(x)+g(y)
					\right.\right.
					\\
					&\quad\left.\left.-\lambda\,\operatorname{exp}\left(\frac{f(x)+g(y)-c(x,y)}{\lambda}\right)\right)\,\operatorname{d}\!\mu\otimes\nu(x,y)\right)+\lambda.
				\end{aligned} 
			\end{equation}
			If we assume $\int_{\mathcal{X}}g\,\operatorname{d}\!\nu=0$, the optimal solutions $\left(f_{\mu,\nu}^{\lambda},g_{\mu,\nu}^{\lambda}\right)$ of the maximization problem \eqref{Eqn:DualEOT}, known as entropic Kantorovich potentials, are unique $\mu\otimes\nu$-a.s. Define the \emph{duality mapping} as $\left(f_{\mu,\nu}^{\lambda},g_{\mu,\nu}^{\lambda}\right)=\psi(\mu,\nu)$. The map $\psi$ is also the solution mapping of the system of equations
			\begin{equation}\label{Eqn:ImplicitRelationPotentials}
				\begin{aligned}
					&f=-\lambda\,\log\left(\int_{\mathcal{X}}\operatorname{exp}\left(\frac{g(y) - c(\cdot, y)}{\lambda}\right)\,\operatorname{d}\!\nu(y)\right),
					\\&g=-\lambda\,\log\left(\int_{\mathcal{X}}\operatorname{exp}\left(\frac{f(x)-c(x,\cdot)}{\lambda}\right)\,\operatorname{d}\!\mu(x)\right).
				\end{aligned}
			\end{equation}
			Additionally, the Kantorovich potentials satisfy
			\begin{equation}\label{Eqn:Densitypi}
				\frac{\operatorname{d}\!\pi_{\mu,\nu}^{\lambda}}{\operatorname{d}\!\mu\otimes\nu}(x,y)=\operatorname{exp}\left(\frac{f_{\mu,\nu}^{\lambda}(x)+g_{\mu,\nu}^{\lambda}(y)-c(x,y)}{\lambda}\right).
			\end{equation}
			Call
			\begin{equation*}
				\xi(f,g)(x,y)=\exp\left(\frac{f(x)+g(y)-c(x,y)}{\lambda}\right),\quad x,y\in\mathcal{X},\ f,g\in\ell^{\infty}(\mathcal{X}).
			\end{equation*}
			Hence, under this notation the Radon-Nykodim derivative in \eqref{Eqn:Densitypi} becomes $\frac{\operatorname{d}\!\pi_{\mu,\nu}^{\lambda}}{\operatorname{d}\!\mu\otimes\nu}=\xi\left(f_{\mu,\nu}^{\lambda},g_{\mu,\nu}^{\lambda}\right)$.
		\subsubsection*{The EOT-plan and its statistical applications}
			As the EOT-plan $\pi_{\mu,\nu}^{\lambda}$, i.e. the solution of \eqref{Eqn:EOT}, approximates the most cost-efficient matching between $\mu$ and $\nu$ it is a useful descriptor for the spatial proximity of particle assemblies. In the context of super-resolution cell microscopy (for a recent review we refer to the Nature Photonics editorial \cite{NaturePhotonicsSuper-resolution2025}), this approach is particularly appealing (see \cite{Klatt2020}). Here, EOT has been utilized for protein matching, where probability measures $\mu$ and $\nu$ model the initial and target molecular assemblies in compartments of cells, respectively. In order to visualize the amount of mass which has to be transported at a certain physical scale, in \cite{Klatt2020} the EOT curve 
			\begin{equation}\label{Eqn:ClassicalColocFn}
				\begin{aligned}
					\varphi(\mu,\nu)(t)&=\int_{\mathcal{X}^{2}}\mathbf{1}_{\{c\leq t\}}(x,y)\,\operatorname{d}\!\pi_{\mu,\nu}^{\lambda}(x,y)
					\\
					&=\int_{\mathcal{X}^{2}}\mathbf{1}_{\{c\leq t\}}(x,y)\,\xi\left(f_{\mu,\nu}^{\lambda},g_{\mu,\nu}^{\lambda}\right)(x,y)\,\operatorname{d}\!\mu\otimes\nu(x,y)
					\\
					&=\int_{\mathcal{X}^{2}}\mathbf{1}_{\{c\leq t\}}(x,y)\,\operatorname{exp}\left(\frac{f_{\mu,\nu}^{\lambda}(x)+g_{\mu,\nu}^{\lambda}(y)-c(x,y)}{\lambda}\right)\,\operatorname{d}\!\mu\otimes\nu(x,y),
				\end{aligned}
			\end{equation}
			was introduced on a finite ground space. In order to estimate this quantity from data, in our general setup we assume that we sample from two Borel probability measures $\mu$ and $\nu$ on $\mathcal{X}$: $X_{1},\ldots,X_{m}\overset{\operatorname{i.i.d.}}{\sim}\mu$ and $Y_{1},\ldots,Y_{n}\overset{\operatorname{i.i.d.}}{\sim}\nu$ are independent, where $\sim$ symbolizes \enquote{distributed as}. We denote the empirical measures associated to those samples as
			\begin{equation}\label{Eqn:EmpMeas}
				\mu_{m}=\frac{1}{m}\,\sum_{i=1}^{m}\delta_{X_{i}},\quad\nu_{n}=\frac{1}{n}\,\sum_{i=1}^{n}\delta_{Y_{i}},
			\end{equation}
			where $\delta_{a}$ stands for \emph{Dirac delta} centered at $a\in\mathcal{X}$. The plug-in principle suggests the sample or empirical counterpart of \eqref{Eqn:ClassicalColocFn}
			\begin{equation}\label{Eqn:EmpClassicalColocFn}
				\varphi\left(\mu_{m},\nu_{n}\right)(t)=\int_{\mathcal{X}^{2}}\mathbf{1}_{\{c\leq t\}}(x,y)\,\operatorname{d}\!\pi_{\mu_{m},\nu_{n}}^{\lambda}(x,y),
			\end{equation}
			as an (natural) estimator (see Figure \ref{fig:TomMicsection} for an illustration). We will refer to that as the \emph{empirical EOT colocalization curve}.
			
			Since $\pi_{\mu_{m},\nu_{n}}^{\lambda}$ is a discrete measure, observe that the integral in \eqref{Eqn:EmpClassicalColocFn} is actually a sum, which can be efficiently computed in $\mathcal{O}(m\,n)$ operations, provided the EOT-plan $\pi_{\mu_{m},\nu_{n}}^{\lambda}$ has been pre-computed.
	\subsection{Main contributions}
		While risk bounds for the empirical EOT colocalization curve and other functionals of the EOT-plan are relatively easy to derive from existing theory (see references above), a valid theory for its weak convergence and corresponding confidence bands still remains elusive, and will be provided in this paper (see Theorem \ref{Th:ColocConv} and Corollary \ref{Cor:Th:ClassicColocConv}). Further examples which are covered by our approach include the well known Zolotarev metric (\cite{Zolotarev1983}, see Subsection \ref{Subsec:Assumptions}). To this end, we present a framework for weak convergence of EOT-based kernel functionals,
		\begin{equation}\label{Eqn:GeneralizedColocMap}
			\Phi\left(\mu_{m},\nu_{n}\right)(u)=\int_{\mathcal{X}^{2}}u(x,y)\,\operatorname{d}\!\pi_{\mu_{m},\nu_{n}}^{\lambda}(x,y),
		\end{equation}
		for certain classes of kernels $u$ (see Definition \ref{Def:GeneralizedColoc}), including $\varphi\left(\mu_{m},\nu_{n}\right)$ in \eqref{Eqn:EmpClassicalColocFn}.
	\subsection{Theory}\label{Subsec:Theory}
		Before presenting the main results, we introduce notation. Given a topological space $\mathfrak{X}$ endowed with the Borel $\sigma$-field $\mathcal{B}_{\mathfrak{X}}$, we will denote the space of finite signed Borel measures on $\mathfrak{X}$ by $\mathcal{M}(\mathfrak{X})$. Further, $\mathcal{P}(\mathfrak{X})$ stands for the subset of $\mathcal{M}(\mathfrak{X})$ consisting of probability measures. Throughout this paper, $\mathfrak{X}=\mathcal{X},\mathcal{X}^{2}$. $\mathcal{F}$ will denote classes of measurable functions on $\mathfrak{X}$ to be specified later. Usual classes of functions that appear in literature are subsets of $\mathcal{C}_{\operatorname{b}}(\mathfrak{X})$, the space of bounded and continuous functions on $\mathfrak{X}$; $\mathcal{C}_{\operatorname{b}}^{s}(\mathfrak{X})$, the subset of $\mathcal{C}_{\operatorname{b}}(\mathfrak{X})$ consisting of $s$-Hölder continuous functions. We use the convention that if $s>1$, derivatives are also bounded. The unit ball of $\mathcal{C}_{\operatorname{b}}(\mathfrak{X})$ (with respect to the supremum norm), is denoted by $\mathcal{C}_{\operatorname{b}}(\mathfrak{X})_{1}$. Further, $\mathcal{C}_{\operatorname{b}}^{s}(\mathfrak{X})_{1}=\mathcal{C}_{\operatorname{b}}^{s}(\mathfrak{X})\cap\mathcal{C}_{\operatorname{b}}(\mathfrak{X})_{1}$. Given $\eta\in\mathcal{P}(\mathfrak{X})$ and provided $\mathcal{F}\subseteq\operatorname{L}^{1}(\eta)$, the space of integrable functions with respect to the measure $\eta$; $\eta$ is identified as an element of the space of $\ell^{\infty}(\mathcal{F})$, the Banach space of bounded functionals endowed with the supremum norm. In addition, as it is usual in the empirical processes literature (see \cite[Section 2]{van_der_Vaart&Wellner2023}), we will use the notation $\eta(f)=\int_{\mathcal{X}}f\,\operatorname{d}\!\eta$ for $f\in\mathcal{F}$.
		
		Let $X_{1},\ldots,X_{m}\overset{\operatorname{i.i.d.}}{\sim}\mu$ and $Y_{1},\ldots,Y_{n}\overset{\operatorname{i.i.d.}}{\sim}\nu$ be independent samples. We denote the associated empirical processes as $\mathbb{G}_{m}^{\mu}=\sqrt{m}\,\left(\mu_{m}-\mu\right)$ and $\mathbb{G}_{n}^{\nu}=\sqrt{n}\,\left(\nu_{n}-\nu\right)$, respectively. Further, define the \emph{joint empirical process} as
		\begin{equation}\label{Eqn:JointEmpProcess}
			\left(\begin{array}{c}
				\sqrt{\frac{n}{m+n}}\,\mathbb{G}_{m}^{\mu}
				\\
				\sqrt{\frac{m}{m+n}}\,\mathbb{G}_{n}^{\nu}
			\end{array}\right)=\sqrt{\frac{m\,n}{m+n}}\,\left(\begin{array}{c}
				\mu_{m}-\mu
				\\
				\nu_{n}-\nu
			\end{array}\right),
		\end{equation}
		that is, the Cartesian product of the empirical processes.
		
		As well as measures, the empirical processes will be considered as a (random) element of the space of functionals over a class of integrable functions. In particular, in this work a natural space to embed the product empirical process into is $\ell^{\infty}\left(\mathcal{F}_{\mu}\right)\times\ell^{\infty}\left(\mathcal{F}_{\nu}\right)$ where
		\begin{equation}\label{Eqn:FmuFnu}
			\begin{aligned}
				\mathcal{F}_{\mu}&=\left\{\int_{\mathcal{X}}u(\cdot,y)\,\xi\left(f_{\mu,\nu}^{\lambda},g_{\mu,\nu}^{\lambda}\right)(\cdot,y)\,\operatorname{d}\!\nu(y):u\in\mathcal{F}\right\},
				\\
				\mathcal{F}_{\nu}&=\left\{\int_{\mathcal{X}}u(x,\cdot)\,\xi\left(f_{\mu,\nu}^{\lambda},g_{\mu,\nu}^{\lambda}\right)(x,\cdot)\,\operatorname{d}\!\mu(x):u\in\mathcal{F}\right\}.
			\end{aligned}
		\end{equation}
		with $\mathcal{F}\subseteq\operatorname{L}^{1}(\mu\otimes\nu)$ to be specified later.
		
		Given a probability measure $\eta$ and a class of functions $\mathcal{F}$, it is said that \emph{$\mathcal{F}$ enjoys the Donsker property for $\eta$} or that \emph{$\mathcal{F}$ is $\eta$-Donsker class} if the empirical process $\mathbb{G}_{j}^{\eta}$ converges weakly in $\mathcal{F}$ when $j\longrightarrow\infty$. We denote it as $\mathbb{G}_{j}^{\eta}\rightsquigarrow\mathbb{G}^{\eta}$ where $\rightsquigarrow$ stands for \enquote{weak convergence}. The limit process $\mathbb{G}^{\eta}$ is a centered Gaussian process with covariance structure
		\begin{equation*}
			\mathbb{E}\left(\mathbb{G}^{\eta}\left(\varphi_{1}\right)\,\mathbb{G}^{\eta}\left(\varphi_{2}\right)\right)=\eta\left(\varphi_{1}\,\varphi_{2}\right)-\eta\left(\varphi_{1}\right)\,\eta\left(\varphi_{2}\right),
		\end{equation*}
		for $\varphi_{1},\varphi_{2}\in\mathcal{F}$. Further, $\mathbb{G}^{\eta}$ has $\eta$-almost-surely uniformly continuous paths (see \cite[Part 2]{van_der_Vaart&Wellner2023}). $\mathbb{G}^{\eta}$ will also be referred as \emph{$\eta$-Brownian bridge}. If the class $\mathcal{F}$ enjoys the Donsker property for every measure $\eta\in\mathcal{P}(\mathcal{X})$, it is said that $\mathcal{F}$ is a \emph{universal Donsker class}.
		
		Next, we list here the assumptions for later reference.
		\begin{description}
			\labitem{(LCm)}{Itm:LCm} The ground space $\mathcal{X}$ is Polish and locally compact.
			\labitem{(Con)}{Itm:Con} The cost $c:\mathcal{X}^{2}\longrightarrow[0,\infty)$ is bounded and continuous.
			\labitem{(Dnk)}{Itm:Dnk} Given $\mu,\nu\in\mathcal{P}(\mathcal{X})$, for some $s>0$, the class $\mathcal{C}_{\operatorname{b}}^{s}(\mathcal{X})_{1}$ is $\mu$ and $\nu$-Donsker. The classes $\mathcal{F}_{\mu}$ and $\mathcal{F}_{\nu}$ defined in \eqref{Eqn:FmuFnu} are $\mu$,$\nu$-Donsker, respectively.
			\labitem{(Bal)}{Itm:Bal} There exists $\tau\in(0,1)$ such that $\rfrac{m}{(m+n)}\longrightarrow\tau$ when $m,n\longrightarrow\infty$. 
			\labitem{(Bnd)}{Itm:Bnd} The class $\mathcal{F}$ is uniformly bounded, that is, there exists $C>0$ such that for all $u\in\mathcal{F}$ we have that $\|u\|_{\infty}<C$.
		\end{description}
		These conditions will be further discussed in Section \ref{Subsec:Assumptions}. Under \ref{Itm:LCm}, \ref{Itm:Con} and \ref{Itm:Bal}, we prove that EOT kernel functions in \eqref{Eqn:GeneralizedColocMap} converges weakly to a certain Gaussian process (see Section \ref{Sec:AsympColoc}). This enables the construction of asymptotic confidence bands for the EOT colocalization curve in \eqref{Eqn:ClassicalColocFn} as well as other supremum-type metrics beyond the uniform metric, such us Zolotarev metrics (see \cite{Zolotarev1983}). Our proof is based on the careful analysis of Hadamard differentiability of the kernel functional $\Phi$. For practical purposes, it is worth to mention that since the expressions of the limit is a linear transformation of the product $\mu$-$\nu$-Brownian bridge, the limit distribution can be approximated by a classic bootstrap scheme (see Section \ref{Sec:StatApplications}). We finally point out that so far the samples in \eqref{Eqn:JointEmpProcess} are independent. Nevertheless, it is worth to emphasize that our strategy of proof allows for straightforward generalization of the limiting behaviour of
		\begin{equation*}
			\sqrt{\frac{m\,n}{m+n}}\,\left(\Phi\left(\mu_{m},\nu_{n}\right)-\Phi(\mu,\nu)\right),
		\end{equation*}
		in various directions, such as for dependent observations, provided the weak convergence of \eqref{Eqn:JointEmpProcess} can be shown; see for instance \cite[Section 4]{Carcamo2020} for copula process or \cite{Dette&Kokot2022} for certain temporal dependencies structures.
	\subsection{Data analysis, computation and software}
		\subsubsection*{Confidence bands}
			The main result  (Theorem \eqref{Th:ColocConv}) in Section \ref{Sec:AsympColoc} provides asymptotic uniform confidence bands for kernel functions satisfying \ref{Itm:Bnd}, including the EOT colocalization curves in \eqref{Eqn:ClassicalColocFn}, that is, for $\alpha \in [0,1]$, 
			\begin{equation}\label{Eqn:confidenceinterval}
				\left[\varphi\left(\mu_{m},\nu_{n}\right)-\sqrt{\frac{m+n}{m\,n}}\,q_{1-\alpha},\varphi\left(\mu_{m},\nu_{n}\right)+\sqrt{\frac{m+n}{m\,n}}\,q_{1-\alpha}\right]
			\end{equation}
			is an asymptotic $1-\alpha$ uniform confidence band for the EOT colocalization curve, where $q_{1-\alpha}$ denotes the $1-\alpha$ quantile of $\left\|\varphi_{(\mu,\nu)}^{\prime}\left(\mathbb{G}^{\mu},\mathbb{G}^{\nu}\right)\right\|_{\infty}$. Here, $\varphi_{(\mu,\nu)}^{\prime}$ is the Hadamard directional derivative of $\varphi$ in \eqref{Eqn:EmpClassicalColocFn} at the point $(\mu,\nu)$ acting on the limit process (see Corollary \ref{Cor:Th:ClassicColocConv}).
			
\noindent In Section \ref{Sec:EmpRes} we discuss two applications: 

\textit{Multimodal distributions (Section \ref{Subsec:multimodal}) } We consider a scenario where a source cloud of points is transported to a target cloud of points, corresponding to three clusters at different cost/distances from the source.  We illustrate this with two examples: the first with Gaussian distributed points in a plane with cost equal to the Euclidean distance, and the second with points on a circle with cost equal to the geodesic distance. The corresponding EOT confidence bands capture, in a visually and statistically sound manner, the three major spatial scales at which mass is transported.

\textit{Super-resolution microscopy (Section \ref{SubSec:RealDataAnalysis})}  In Figure \ref{fig:TomMicsection} (upper panel) assemblies of two mitochondrial proteins Tom20 and Mic60 are displayed. They have been recorded using STimulated Emission Depletion (STED) super-resolution microscopy (\cite{hell2007far}), a methodology which allows for discerning such proteins at a resolution far below the Abbe diffraction barrier delimiting conventional light microscopy to discern objects below a distance of about 250 - 500 nm. Besides the randomness inherent to the photon emission process of the photons recorded in the image, we suggest the use of these confidence bands as a tool for the validation of fast computational schemes based on subsampling. In the lower panel of Figure \ref{fig:TomMicsection} we display such confidence bands in \eqref{Eqn:confidenceinterval} for different levels of confidence $1-\alpha$, where $q_{1-\alpha}$ has been approximated by bootstrap. For further details, see Section \ref{SubSec:RealDataAnalysis}.

\subsubsection*{Computational aspects}
For all computational tasks in Section 4.1, we use the \emph{POT} package (\cite{flamary2021pot}) with \verb!Python! version 3.11.12, while for the computation of the optimal transport plan between STED images, we use the \verb!Julia! package \emph{MuSink} (\cite{Musink2025Staudt}) with \verb!Julia! version 1.11.6. This efficiently implements EOT between measures on uniform two-dimensional grids (i.e., well suited for images) via Sinkhorn iterations. All the code and data used to generate the figures in this document are available in the corresponding \verb!GitLab! repository (available at \url{https://gitlab.gwdg.de/arenasvelill/EEOT-Colocalization/}). To accelerate computation, we leverage the CUDA support provided by both packages.

		\begin{figure}[!htbp]
	\centering
	\includegraphics[width=0.9\textwidth]{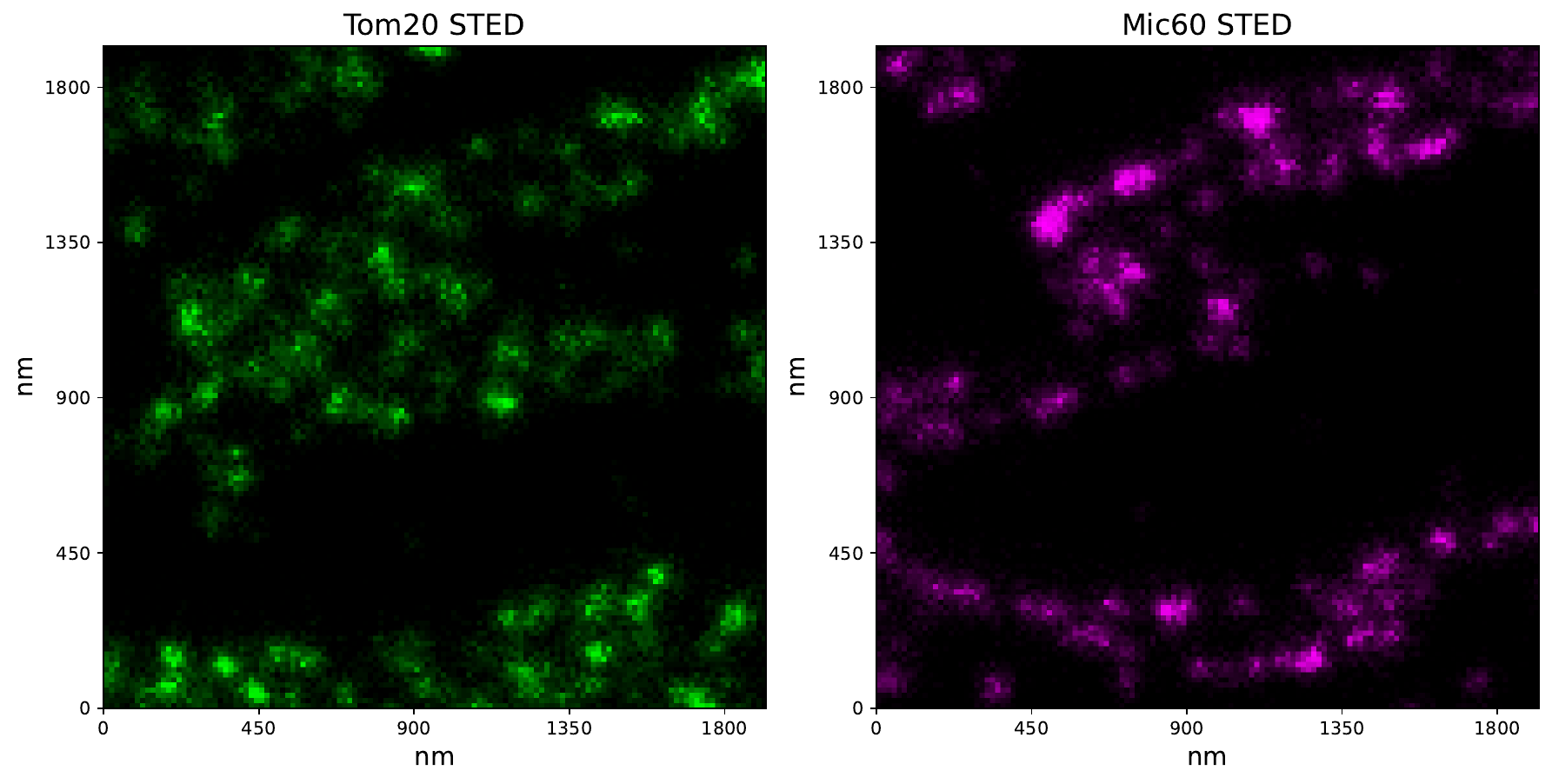}
	\includegraphics[width=0.8\textwidth]{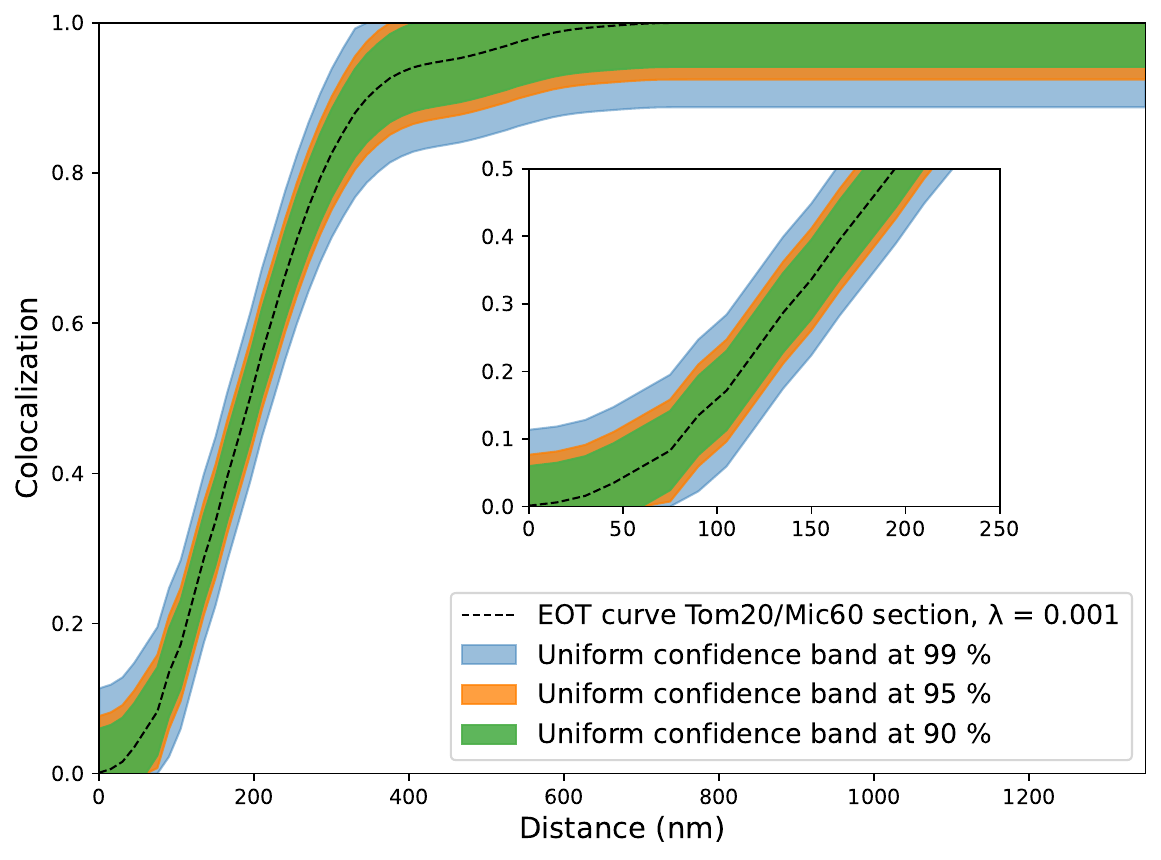}
	\caption{Top: \textit{STED nanoscopy images of Tom20 and Mic60:} Section of size $128 \times 128$ pixels generated using data from the Optimal Transport Colocalization (OTC) dataset in \cite{carla_tameling_2021}. Bottom:  Confidence bands: the dashed line correspond to the EOT curve between the protein assemblies in the above images (with regularization parameter $\lambda = 0.001$), with uniform bootstrap confidence bands at level $1-\alpha$, for $\alpha = 0.1, 0.05, 0.01$, based on the $n$ out of $n$ bootstrap with  $B=2000$ replications. }
	\label{fig:TomMicsection}
\end{figure}

	\subsection{Related work}
		Our work is mainly inspired by \cite{Klatt2020}, where an analysis of the asymptotic behavior of the entropic optimal transport plan (and the functionals thereof) and corresponding confidence intervals has been provided for measures with finite support. Note that $\varphi\left(\mu_{m},\nu_{n}\right)$ can be viewed as an estimator of $\pi_{\mu,\nu}^{\lambda}$ indexed by indicator functions $\mathbf{1}_{\{c\leq t\}}$. As it was seen in \cite{Gonzalez2022} and \cite{Gonzalez&Hundrieser2023}, for every $t\in\left[0,\|c\|_{\infty}\right]$, $\sqrt{n}\,\left(\pi_{\mu_{m},\nu_{n}}^{\lambda}-\pi_{\mu,\nu}^{\lambda}\right)\left(\mathbf{1}_{\{c\leq t\}}\right)$ converges weakly to a centered Gaussian distribution. That result is also obtained as a consequence of asymptotic theory to estimate EOT potentials in \eqref{Eqn:ImplicitRelationPotentials}, see \cite{Goldfeld2024}. Here, these results are extended to a uniform limit on the set of indicators and dropping the restriction on the support of the underlying measures. Inspired by the methodology proposed in \cite{Goldfeld2024}, differentiability results of \cite{Carlier&Laborde2020} are extended (see Theorem \ref{Lema:ourDiffDens}), which allows the application of the functional delta method to obtain the uniform weak limit of kernel functions. This includes colocalization curves (uniform convergence on the set of indicators). Finally, we refer to \cite{Liu2025} for a careful analysis of the unregularized optimal plan $\pi_{\mu,\nu}$ by the approximation of the EOT plan $\pi_{\mu,\nu}^{\lambda}$. In particular, the authors emphasize that regularization of the transport problem can introduce a bias. In Subsection \ref{SubSec:RealDataAnalysis}, we provide empirical evidence showing that this bias is asymptotically negligible when constructing uniform confidence bands for EOT in our application.
	\subsection{Outline of the paper}
		The structure of the paper unfolds as follows: In Section \ref{Sec:LimitDist} we introduce the technicalities necessary for our purposes. In particular, we generalized the concept of colocalization function to kernel functions. Furthermore, it contains the main theoretical result of this paper: The (weak) limit of the kernel-type estimator, including the colocalization function. This result is a consequence of the delta method and Hadamard directional differentiability of the functionals involved. Details are given in the Supplementary material. Section \ref{Sec:StatApplications} is devoted to explain the statistical methodology derived from the theoretical results of Section \ref{Sec:LimitDist}. Those tools are applied in Section \ref{Sec:EmpRes}, where several datasets (both simulated and real) are analyzed.
\section{Limit distribution of empirical EOT kernel function}\label{Sec:LimitDist}
	In this section we investigate the limit distribution of empirical colocalization curves. Firstly, introduce the following generalization of the colocalization curve.
	\begin{Def}\label{Def:GeneralizedColoc}
		Given $\mu,\nu\in\mathcal{P}(\mathcal{X})$, $\lambda>0$, and $u\in\operatorname{L}^{1}(\mu\otimes\nu)$, the \emph{EOT kernel function} is defined as
		\begin{equation*}
			\begin{aligned}
				\Phi(\mu,\nu)(u)&=\pi_{\mu,\nu}^{\lambda}(u)=\int_{\mathcal{X}^{2}}u(x,y)\,\operatorname{d}\!\pi_{\mu,\nu}^{\lambda}(x,y)
				\\
				&=\int_{\mathcal{X}^{2}}u(x,y)\,\xi\left(f_{\mu,\nu}^{\lambda},g_{\mu,\nu}^{\lambda}\right)(x,y)\,\operatorname{d}\!\mu\otimes\nu(x,y).
			\end{aligned}
		\end{equation*}
	\end{Def}
	Definition \ref{Def:GeneralizedColoc} may be read in the following terms: The operator $\Phi$ (or $\operatorname{EOT}$) maps the marginal measures $\mu$ and $\nu$ to a new space (functionals over a class $\mathcal{F}\subseteq\operatorname{L}^{1}(\mu\otimes\nu)$). By definition of entropic optimal transport \eqref{Eqn:EOT}, $\pi_{\mu,\nu}^{\lambda}\ll\mu\otimes\nu$ with Radon-Nikodym derivative \eqref{Eqn:Densitypi}. Since \eqref{Eqn:Densitypi} is bounded away from zero and infinity (see \cite[Lemma 1]{Goldfeld2024} or Lemma \ref{Lema:BoundsPotentials} in the Supplementary material), we conclude also that $\pi_{\mu,\nu}^{\lambda}\gg\mu\otimes\nu$, so $\operatorname{L}^{1}(\mu\otimes\nu)=\operatorname{L}^{1}\left(\pi_{\mu,\nu}^{\lambda}\right)$.
	
	A particular examples are colocalization curves given in \eqref{Eqn:ClassicalColocFn}, where $\pi_{\mu,\nu}^{\lambda}$ denotes the solution of \eqref{Eqn:EOT}. We write $\varphi(\mu,\nu)(t)=\pi_{\mu,\nu}^{\lambda}\left(\mathbf{1}_{\{c\leq t\}}\right)$. Further, if we denote by $\mathcal{F}_{I,c}=\left\{\mathbf{1}_{\{c\leq t\}}:t\in[0,\|c\|_{\infty}\right\}$, it is straightforward that $\pi_{\mu,\nu}^{\lambda}\in\ell^{\infty}\left(\mathcal{F}_{I,c}\right)$.
	\subsection{Assumptions}\label{Subsec:Assumptions}
		Before showing the main theoretical results of this work, we will discuss the assumptions stated in Section \ref{Subsec:Theory}. In particular, we will see that these are satisfied by a considerable amount of scenarios that appear in applications (see Section \ref{Sec:EmpRes}).
		
		First consider \ref{Itm:LCm} and \ref{Itm:Con} jointly. When $\mathcal{X}$ is a Polish space and the cost $c$ is bounded, \eqref{Eqn:EOT} (equivalently \eqref{Eqn:DualEOT}) has a unique solution, in particular, the optimal plan $\pi_{\mu,\nu}^{\lambda}$ exists (see \cite[Theorem 4.2]{Nutz2021}). \ref{Itm:LCm} and \ref{Itm:Con} constitute sufficient conditions to work in this framework. The former, \ref{Itm:LCm}, is satisfied by any closed subset of any finite dimensional complete topological vector spaces. In the scenarios presented in Section \ref{Sec:EmpRes}, it is always satisfied. Other authors such as in \cite{Goldfeld2024} or \cite{Gonzalez2022} require that $\mathcal{X}$ to be compact. Nevertheless, to the best of our knowledge, this property is too restrictive. Indeed, local compactness in \ref{Itm:LCm} is a sufficient condition to embed $\mathcal{P}(\mathcal{X})$ in the space of Radon measures on $\mathcal{X}$ (endowed with the topology of the total variation norm, see \cite[Chapter 7]{Folland2013}). Combined to the continuity of the cost $c$ in \ref{Itm:Con} provides sufficient condition for the Hadamard directional differentiability of $\Phi$ in Definition \ref{Def:GeneralizedColoc} (see Theorem \ref{Th:DiffPhi} in the Supplementary material).
		
		When \ref{Itm:Bal} is satisfied, we say that the \emph{sampling scheme is balanced}. Together with \ref{Itm:Dnk}, it implies that
		\begin{equation}\label{Eqn:LimitBB}
			\left(\begin{array}{c}
				\sqrt{\frac{n}{m+n}}\,\mathbb{G}_{m}^{\mu}
				\\
				\sqrt{\frac{m}{m+n}}\,\mathbb{G}_{n}^{\nu}
			\end{array}\right)\rightsquigarrow\left(\begin{array}{c}
				\sqrt{1-\tau}\,\mathbb{G}^{\mu}
				\\
				\sqrt{\tau}\,\mathbb{G}^{\nu}
			\end{array}\right),
		\end{equation}
		in $\ell^{\infty}\left(\mathcal{F}_{\mu}\right)\times\ell^{\infty}\left(\mathcal{F}_{\nu}\right)$. \ref{Itm:Dnk} and other Donsker-like properties are well studied in literature. For sufficient conditions on the Donsker property for $\mathcal{C}_{\operatorname{b}}^{s}(\mathfrak{X})_{1}$ we refer to \cite[Theorems 2.7.1, 2.7.2 and 2.7.3]{van_der_Vaart&Wellner2023}. The next result shows that for classes of functions $\mathcal{F}\subseteq\operatorname{L}^{1}(\mu\otimes\nu)$ that usually appear in this context, $\mathcal{F}_{\mu}$ and $\mathcal{F}_{\nu}$ enjoy the Donsker property.
		\begin{Th}\label{Th:Donskerity}
			Assume \ref{Itm:Bnd}, and either
			\begin{enumerate}
				\item The class $\mathcal{F}$ is the class of indicators $\mathcal{F}_{I,c}=\left\{\mathbf{1}_{\{c\leq t\}}:t\in\left[0,\|c\|_{\infty}\right]\right\}$, or
				\item Take $s\geq0$, $\mathcal{X}\subset\mathbb{R}^{l}$, $l\in\mathbb{N}\setminus\{0\}$, bounded, convex and with non-void interior and $c\in\mathcal{C}_{\operatorname{b}}^{s}\left(\mathcal{X}^{2}\right)$. The class $\mathcal{F}$ is $\mathcal{C}_{\operatorname{b}}^{s}\left(\mathcal{X}^{2}\right)_{1}$.
			\end{enumerate}
			Then, $\mathcal{F}_{\mu}$ and $\mathcal{F}_{\nu}$ enjoy the Donsker property.
		\end{Th}
		It is worth to make some comments about the classes of functions that appear in the previous theorem. The integral probability metric generated by $\mathcal{F}_{I,c}$ in the space of probability measures $\mathcal{P}\left(\mathcal{X}^{2}\right)$ is the uniform metric. The requirement $c\in\mathcal{C}_{\operatorname{b}}^{s}\left(\mathcal{X}^{2}\right)$ is quite common in literature (see for instance \cite{Genevay2019} or \cite{Gonzalez2022}) and provides sufficient conditions for $f_{\mu,\nu}^{\lambda},g_{\mu,\nu}^{\lambda}\in\mathcal{C}_{\operatorname{b}}^{s}\left(\mathcal{X}^{2}\right)$. Furthermore, the integral probability metric
		\begin{equation*}
			\zeta\left(\eta_{1},\eta_{2}\right)=\sup_{u\in \mathcal{C}_{\operatorname{b}}^{s}\left(\mathcal{X}^{2}\right)_{1}}\left(\left|\int_{\mathcal{X}}u\,\operatorname{d}\!\left(\eta_{1}-\eta_{2}\right)\right|\right),
		\end{equation*}
		is known as the Zolotarev $s$-metric (see \cite{Zolotarev1983}).
		
		Finally, we remark that \ref{Itm:Bnd} can be weakened to the assumption that the envelope of the class $\mathcal{F}$ is in $\operatorname{L}^{1}(\mu\otimes\nu)$.
	\subsection{Asymptotic result for colocalization curves}\label{Sec:AsympColoc}
		In this subsection we will focus on the main result of this paper: The asymptotic behaviour of the plug-in estimator $\Phi\left(\mu_{m},\nu_{n}\right)$ of the EOT kernel function $\Phi(\mu,\nu)$ in Definition \ref{Def:GeneralizedColoc}. Recall that given independent samples $X_{1},\ldots,X_{m}\sim\mu$ and $Y_{1},\ldots,Y_{n}\sim\nu$, we defined in \eqref{Eqn:JointEmpProcess} the joint empirical process. Let us denote $-\mu+\mathcal{P}(\mathcal{X})=\left\{-\mu+\eta:\eta\in\mathcal{P}(\mathcal{X})\right\}$ and similarly for $\nu$.
		\begin{Th}[Weak limit of empirical EOT kernel functions]\label{Th:ColocConv}
			Assume \ref{Itm:LCm}, \ref{Itm:Con}, \ref{Itm:Dnk}, \ref{Itm:Bal} and \ref{Itm:Bnd}. Then for $\Phi$ in \eqref{Eqn:GeneralizedColocMap}
			\begin{equation*}
				\sqrt{\frac{m\,n}{m+n}}\,\left(\Phi\left(\mu_{m},\nu_{n}\right)-\Phi(\mu,\nu)\right)\rightsquigarrow\Phi_{(\mu,\nu)}^{\prime}\left(\sqrt{1-\tau}\,\mathbb{G}^{\mu},\sqrt{\tau}\,\mathbb{G}^{\nu}\right),
			\end{equation*}
			where $\rightsquigarrow$ denotes weak convergence of distributions, $\mathbb{G}^{\mu}$ and $\mathbb{G}^{\nu}$ are independent $\mu$, $\nu$-Brownian bridges on $\mathcal{F}_{\mu}$ and $\mathcal{F}_{\nu}$, respectively; and $\Phi_{(\mu,\nu)}^{\prime}$ is the Hadamard directional derivative of $\Phi$, defined as the functional
			\begin{equation}
				\begin{aligned}
					&\Phi_{(\mu,\nu)}^{\prime}\left(h^{\mu},h^{\nu}\right)(u)=\mathcal{I}_{1}(u)+\mathcal{I}_{2}(u),\quad h^{\mu}\in-\mu+\mathcal{P}(\mathcal{X}),h^{\nu}\in-\nu+\mathcal{P}(\mathcal{X}),\quad u\in\mathcal{F},
					\\
					&\begin{aligned}
						\mathcal{I}_{1}(u)&=\frac{1}{\lambda}\,\int_{\mathcal{X}^{2}}u(x,y)\,\xi\left(f_{\mu,\nu}^{\lambda},g_{\mu,\nu}^{\lambda}\right)(x,y)\,\left(\psi_{(\mu,\nu)}^{\prime}\left(h^{\mu},h^{\nu}\right)^{(1)}(x)\right.
						\\
						&\quad\left.+\psi_{(\mu,\nu)}^{\prime}\left(h^{\mu},h^{\nu}\right)^{(2)}(y)\right)\,\operatorname{d}\!\mu\otimes\nu(x,y),
					\end{aligned}
					\\
					&\mathcal{I}_{2}(u)=\int_{\mathcal{X}^{2}}u(x,y)\,\xi\left(f_{\mu,\nu}^{\lambda},g_{\mu,\nu}^{\lambda}\right)(x,y)\,\operatorname{d}\!\left(\mu\otimes h^{\nu}+h^{\mu}\otimes\nu\right)(x,y),
				\end{aligned}
			\end{equation}
			where the superindex $^{(i)}$ for $i\in\{1,2\}$ denotes the projection on the $i$-th coordinate.
		\end{Th}
		As a consequence of Theorem \ref{Th:ColocConv}, when $\mathcal{F}_{I,c}=\left\{\mathbf{1}_{\{c\leq t\}}:t\in\left[0,\|c\|_{\infty}\right]\right\}$ is taken, we obtain the weak convergence of colocalization curves uniformly in $t\in\left[0,\|c\|_{\infty}\right]$. To this end, note that \ref{Itm:Bnd} is obviously fulfilled and \ref{Itm:Dnk} is given by Theorem \ref{Th:Donskerity}.
		\begin{Cor}[Weak convergence of empirical colocalization curve]\label{Cor:Th:ClassicColocConv}
			Assume \ref{Itm:Con}, \ref{Itm:LCm} and \ref{Itm:Bal}. Then for $\varphi$ in \eqref{Eqn:ClassicalColocFn}
			\begin{equation}\label{Eqn:ClassicalColocConv}
				\sqrt{\frac{m\,n}{m+n}}\,\left(\varphi\left(\mu_{m},\nu_{n}\right)(t)-\varphi(\mu,\nu)(t)\right)\rightsquigarrow\varphi_{(\mu,\nu)}^{\prime}\left(\sqrt{1-\tau}\,\mathbb{G}^{\mu},\sqrt{\tau}\,\mathbb{G}^{\nu}\right)(t),
			\end{equation}
			uniformly on $t\in\left[0,\|c\|_{\infty}\right]$, where
			\begin{equation*}
				\varphi_{(\mu,\nu)}^{\prime}\left(\sqrt{1-\tau}\,\mathbb{G}^{\mu},\sqrt{\tau}\,\mathbb{G}^{\nu}\right)(t)=\Phi_{(\mu,\nu)}^{\prime}\left(\sqrt{1-\tau}\,\mathbb{G}^{\mu},\sqrt{\tau}\,\mathbb{G}^{\nu}\right)\left(\mathbf{1}_{\{c\leq t\}}\right).
			\end{equation*}
		\end{Cor}
\section{Statistical applications}\label{Sec:StatApplications}
	Based on Theorem \ref{Th:ColocConv} and by virtue of the (extended) continuous mapping theorem, we can provide uniform confidence bands for the general EOT kernel function in \eqref{Eqn:GeneralizedColocMap} and in particular for the colocalization function $\varphi(\mu,\nu)$ in \eqref{Eqn:ClassicalColocFn}, recall Corollary \ref{Cor:Th:ClassicColocConv}. A practical difficulty arises as the explicit limit requires the estimation of the EOT potentials in \eqref{Eqn:DualEOT} as well as its derivative with respect to the measures $\mu$ and $\nu$ (see \cite{Gonzalez&Hundrieser2023}, \cite{Goldfeld2024}, and references therein). Alternatively, in this work we propose a different approach to bypass this obstacle: Bootstrapping.
	
	For the sake of completeness, let us firstly describe the bootstrapping procedure. Given independent samples $X_{1},\ldots,X_{m}\sim\mu$ and $Y_{1},\ldots,Y_{n}\sim\nu$ and its associated empirical measures $\mu_{m}$ and $\nu_{n}$ in \eqref{Eqn:EmpMeas} the bootstrap distribution is computed as follows: Take independent samples $X_{1}^{\ast},\ldots,X_{m^{\ast}}^{\ast}\sim\mu_{m}$ and $Y_{1}^{\ast},\ldots,Y_{n^{\ast}}^{\ast}\sim\nu_{n}$, with $m^{\ast},n^{\ast}\in\mathbb{N}$. For simplicity, let's fix $m^{\ast}=m$ and $n^{\ast}=n$. Denote by $\mu_{m}^{\ast}$ and $\nu_{n}^{\ast}$ the (bootstrap) empirical measures associated to the $^{\ast}$ samples and by
	\begin{equation}\label{Eqn:JointBootsProcess}
		\left(\begin{array}{c}
			\sqrt{\frac{n}{m+n}}\,\mathbb{G}_{m}^{\mu,\ast}
			\\
			\sqrt{\frac{m}{m+n}}\,\mathbb{G}_{n}^{\nu,\ast}
		\end{array}\right)=\sqrt{\frac{m\,n}{m+n}}\,\left(\begin{array}{c}
			\mu_{m}^{\ast}-\mu_{m}
			\\
			\nu_{n}^{\ast}-\nu_{n}
		\end{array}\right),
	\end{equation}
	the joint bootstrap process. By \cite[Theorem 3.7.3]{van_der_Vaart&Wellner2023} the process in \eqref{Eqn:JointBootsProcess} converges weakly (conditional to the sample) and the limit is the same as in \eqref{Eqn:LimitBB} given \ref{Itm:Dnk}.
	
	Now, the bootstrap version of \eqref{Eqn:ClassicalColocConv} is defined as $\sqrt{\frac{m\,n}{m+n}}\,\left(\varphi\left(\mu_{m}^{\ast},\nu_{n}^{\ast}\right)-\varphi\left(\mu_{m},\nu_{n}\right)\right)$. Since the Hadamard directional derivative of the map $\Phi$ is linear (see Theorem \ref{Th:ColocConv} or Theorem \ref{Th:DiffPhi} in the Supplementary material), we infer that the bootstrap version of \eqref{Eqn:ClassicalColocConv} converges weakly (conditional to the sample, uniformly) and the limit is the same as the one given in Corollary \ref{Cor:Th:ClassicColocConv} (see \cite[Chapter 3.10]{van_der_Vaart&Wellner2023}). Hence, repeating this $B$ times we can approximate the distribution of $\varphi_{(\mu,\nu)}^{\prime}\left(\sqrt{1-\tau}\,\mathbb{G}^{\mu},\sqrt{\tau}\,\mathbb{G}^{\nu}\right)$.
	
	By definition, the asymptotic $1-\alpha$-confidence band is the set of functions $u\in\ell^{\infty}\left(\left[0,\|c\|_{\infty}\right]\right)$ that satisfy $\left\|u-\varphi\left(\mu_{m},\nu_{n}\right)\right\|_{\infty}\leq\sqrt{\frac{m+n}{m\,n}}\,q_{1-\alpha}$, where
	\begin{equation*}
		\mathbf{P}\left(\left\|\varphi_{(\mu,\nu)}^{\prime}\left(\sqrt{1-\tau}\,\mathbb{G}^{\mu},\sqrt{\tau}\,\mathbb{G}^{\nu}\right)\right\|_{\infty}\leq q_{1-\alpha}\right)=1-\alpha.
	\end{equation*}
	In other words, $q_{1-\alpha}$ is the $1-\alpha$ quantile of the real valued random variable
	\break
	$\left\|\varphi_{(\mu,\nu)}^{\prime}\left(\sqrt{1-\tau}\,\mathbb{G}^{\mu},\sqrt{\tau}\,\mathbb{G}^{\nu}\right)\right\|_{\infty}$. As $\varphi_{(\mu,\nu)}^{\prime}\left(\sqrt{1-\tau}\,\mathbb{G}^{\mu},\sqrt{\tau}\,\mathbb{G}^{\nu}\right)$ can be approximated by its bootstrap version, a direct application of the continuous mapping theorem gives the bootstrap estimator $q_{1-\alpha}^{\ast}$ of $q_{1-\alpha}$ built as the $1-\alpha$-quantile of the random variable $\left\|\sqrt{\frac{m\,n}{m+n}}\,\left(\varphi\left(\mu_{m}^{\ast},\nu_{n}^{\ast}\right)-\varphi\left(\mu_{m},\nu_{n}\right)\right)\right\|_{\infty}$. Thereby, $1-\alpha$ asymptotic confidence bands are computed as the bracket $\left[\varphi\left(\mu_{m},\nu_{n}\right)-\sqrt{\frac{m+n}{m\,n}}\,q_{1-\alpha}^{\ast},\varphi\left(\mu_{m},\nu_{n}\right)+\sqrt{\frac{m+n}{m\,n}}\,q_{1-\alpha}^{\ast}\right]$.
\section{Data Analysis}\label{Sec:EmpRes}

    	\subsection{Colocalization for multimodal distributions} \label{Subsec:multimodal}    
    	
    	Consider a cloud of points that has to be cost efficiently transported to another point cloud which has several different clusters. The colocalization curve gives us useful information about the proportion of points of the original sample that are transported to each cluster.
    	
 As the target points form distinct clusters, the distances or costs between the source points and these clusters vary, while transport to points within a cluster will be approximately of equal cost. Which is well reflected by the associated colocalization curve.

    	 To highlight this feature, we present two simulation examples, each with a distinct distribution of cloud points and cost function. In both cases, we assume that the source and target samples have the same size.

		\textbf{Scenario I.} We consider samples in $\mathbb{R}^2$, where  $X_1, \ldots, X_n$ is a sample from a 2-dimensional standard Gaussian $\mathcal{N}_2(0, \mathbb{I}_2)$, and $Y_1, \ldots, Y_n$ is a sample from a Gaussian mixture $M \mathcal{N} = \sum_{i=1}^3 p_i \mathcal{N}(a_i, \Sigma) $, with the following parameters
		
		\begin{align*}
		\begin{pmatrix} p_1 \\ p_2 \\ p_3\end{pmatrix} = \begin{pmatrix} 0.3 \\ 0.3 \\ 0.4 \end{pmatrix},  \quad a_1 =  \begin{pmatrix} -1 \\ 3 \end{pmatrix},  \quad a_2 = \begin{pmatrix} 10 \\ -1 \end{pmatrix}, \quad a_3 =\begin{pmatrix} 15 \\ 5 \end{pmatrix}, \Sigma = \begin{bmatrix}1 & -0.8 \\-0.8 & 1
    	\end{bmatrix}.
		\end{align*}		 
		
		The cost is chosen as the Euclidean distance. The means of each mixture component were selected at different distances from the mean of the source sample to obtain three distinct components. The selection of the weights and the covariance matrix is aimed at obtaining non-circular shape of the clusters. In Figure \ref{fig:gauss} (left), we display samples of size $n=100$ together with the corresponding empirical EOT plan, for the corresponding EOT curve, see Figure \ref{fig:gauss} (right).
		
		 Following the procedures described in Section \ref{Sec:StatApplications}, we consider three sample sizes $n=100, 500, 1000$ and bootstrap replications $B=1000$. From these, we compute the corresponding empirical EOT curves to generate the bootstrap sample of $\left\|\left(\varphi\left(\mu_{n}^{\ast},\nu_{n}^{\ast}\right)-\varphi\left(\mu_{n},\nu_{n}\right)\right)\right\|_{\infty}$. For each $n$, we generate $B$ independent samples from $\mathcal{N}_2(0, \mathbb{I}_2)$ and  $M \mathcal{N}$, and compute the associated EOT curve to obtain a sample of $\left\|\left(\varphi\left(\mu_{n},\nu_{n}\right)-\varphi\left(\mu,\nu\right)\right)\right\|_{\infty}$. To test the validity of our approach, the underlying population colocalization curve is obtained using a single sample of size $20 \,000$, which serves as ground truth. Corresponding Q-Q plots are displayed in Figure \ref{fig:gauss_qq}. The bootstrap quantile values are presented in Table \ref{table:quantil_gauss}.  Using these values, we build $1-\alpha$ uniform confidence bands (see Figure \ref{fig:gauss_confbands}).

	\begin{table}[h!]
		\begin{center}
			\begin{tabular}{ |c|c|c|c| } 
 			\hline
 			$n$ & 100 & 500 & 1000 \\
 			\hline
			$\widehat{q}_{95}^{\ast}$ & 0.9474 & 0.8486 & 0.8979\\
			\hline
			$\widehat{q}_{95}^{\ast}\sqrt{\frac{2}{n}}$ & 0.1340 & 0.0537 & 0.0402  \\
		 	\hline
			\end{tabular}
		\end{center}
	\caption{Bootstrap quantiles Gaussian vs Gaussian Mixture (scenario I).} 
	\label{table:quantil_gauss}	
	\end{table}

    	\begin{figure}[!htbp]
 	 	\centering
  			\includegraphics[width=\textwidth]{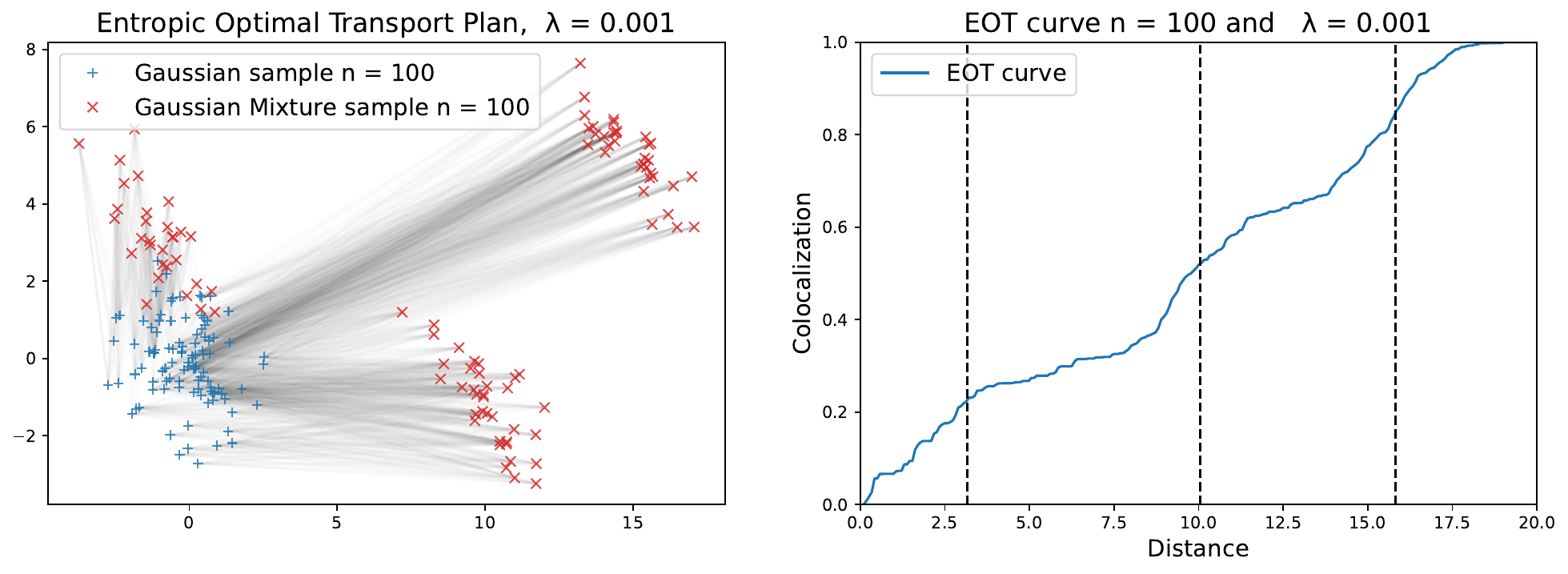}
  		\caption{\text{Left:} \textit{Scenario I: Gaussian vs Gaussian Mixture.}  In blue $+$ the source sample from the Gaussian distribution and in red $\times$ the target sample from the Gaussian mixture. Observe that the sample mixture form three clusters, whose means lie at different distances from the mean of the source distribution. The gray lines represent the EOT plan (with regularization parameter $\lambda = 0.001$) between the samples. \text{Right:} The blue line correspond to the empirical EOT curve between the Gaussian and the Gaussian mixture distributions (with regularization parameter $\lambda = 0.001$). The vertical lines display the distance between the mean of the Gaussian distribution and the means of the components of the Gaussian mixture.}
  		\label{fig:gauss}
  		\end{figure}
  		
    	\begin{figure}[!htbp]
 	 	\centering
  			\includegraphics[width=\textwidth]{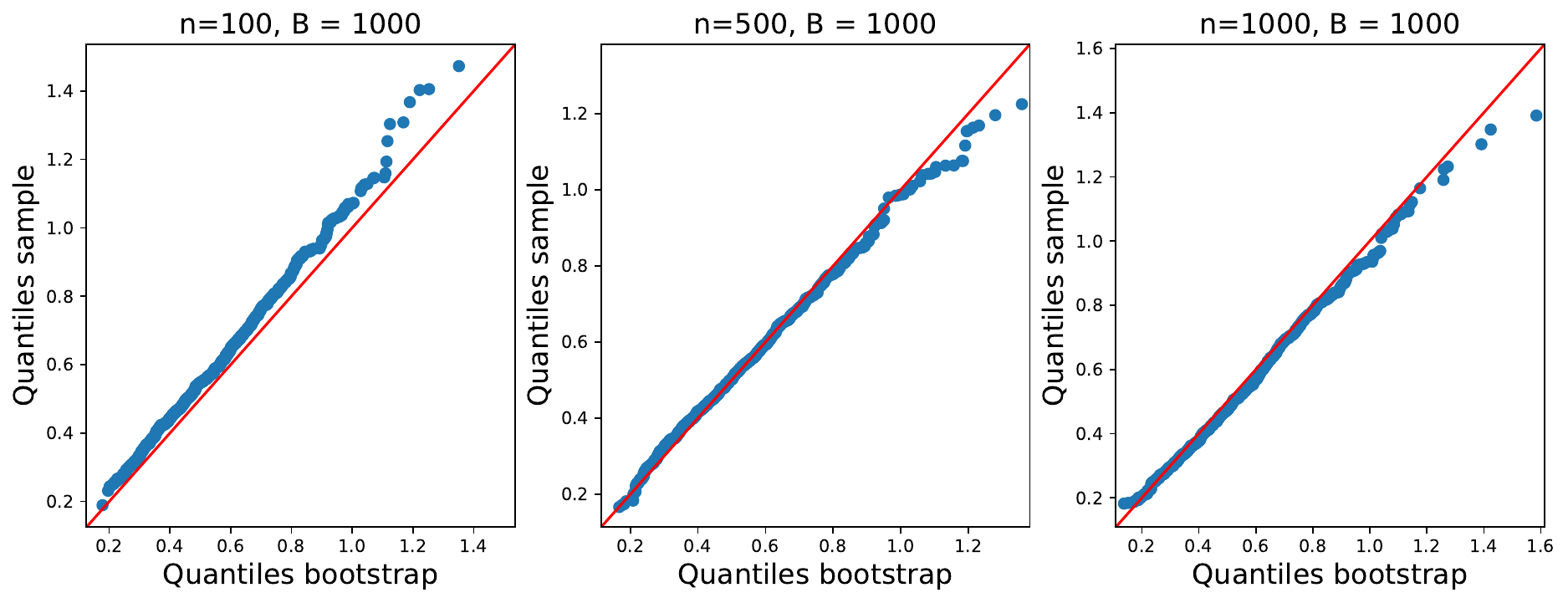}
  		\caption{\textit{Bootstrap accuracy of the EOT limit law for scenario I.} Q-Q plots of the bootstrap distribution of the uniform norm of the empirical EOT curve for three sample sizes $n=100, 500, 1000$ and bootstrap replications $B=1000$. The underlying true EOT curve is obtained from a sample of size $20\,000$, which serves as an approximation of the ground truth.}
  		\label{fig:gauss_qq}
  		\end{figure}

    	\begin{figure}[h!]
			\centering			
			\includegraphics[width=\textwidth]{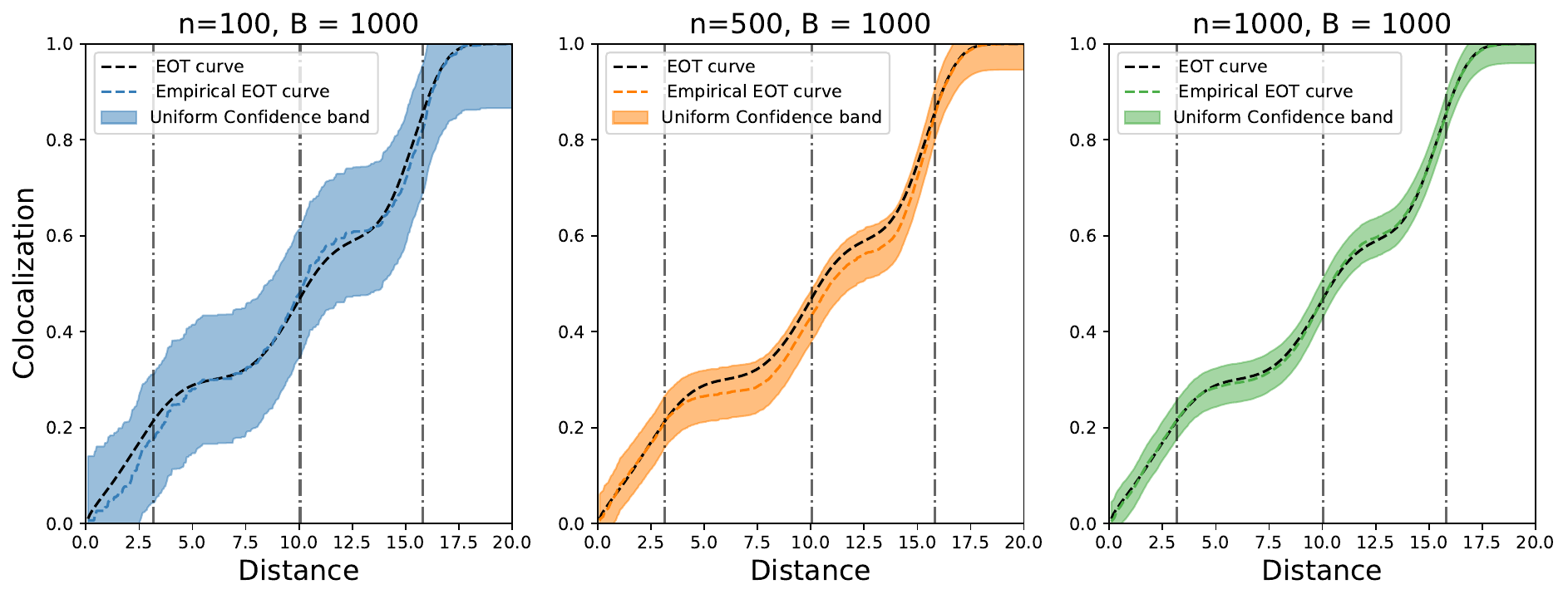}
			\caption{\textit{EOT-Confidence bands for scenario I:}  The dashed line corresponds to the colocalization curve (which serves as ground truth) between the Gaussian and the Gaussian mixture distributions (with sample size $20\,000$ and regularization parameter $\lambda = 0.001$). The dotted lines correspond to the EOT curves obtained from $n=100, 500$ and $1000$ observations, shown with uniform bootstrap confidence bands at level $1-\alpha$. Here $\alpha = 0.05$, based on the $n$ out of $n$ bootstrap with  $B=1000$ replications for $n=100, 500, 1000$. The bands are of the form $\varphi(\mu_n, \nu_n) \pm \widehat{q}_{95}^{\ast} \sqrt{\frac{2}{n}}$ with $\widehat{q}_{95}^{\ast}$ as in the Table \ref{table:quantil_gauss}.}
			\label{fig:gauss_confbands}
		\end{figure}

		 \textbf{Scenario II.} We investigate colocalization curves obtained from data on the sphere $\mathbb{S}_2$, with the geodesic distance as cost function, 
		 \begin{equation*}
		 \label{eq:geodesicdistance}
			c(x,y) = \rho_{\mathbb{S}_2}(x,y) = \mathrm{arccos}( x \cdot y) = \mathrm{arccos}(x_1y_1 + x_2y_2 +x_3y_3), \qquad \forall x, y \in \mathbb{S}_2. 
		 \end{equation*}

		In this case, we assume that the source and target distribution come from a von Mises-Fisher distribution $M_{3}(a, \kappa)$, with probability density function (w.r.t. uniform measure in $\mathbb{S}_2$) of a unit vector $x$ given by
		\begin{equation*}\label{eq:densityvonMF}
		f(x) = \frac{ \kappa^{\rfrac{3}{2}}}{\sqrt{2\pi} I_{\rfrac{1}{2}}(\kappa)}e^{\kappa a^tx}, 
		\end{equation*}
where $a \in \mathbb{S}_2$ is the mean direction, $\kappa >0$  the concentration parameter, and $I_{\rfrac{1}{2}}$ is the modified Bessel function of order $\rfrac{1}{2}$, see \cite{mardia2009directional}.  Let $X_1, X_2, \ldots, X_n$ be an i.i.d sample from a von Mises-Fisher distribution $M_{3}((0,0,1), 50)$ and let $Y_1, \ldots, Y_n$ be an i.i.d sample from a von Mises-Fisher mixture $\sum_{i=1}^3 p_i M_{3}(a_i, \kappa_i) $, with $(p_1, p_2, p_3) = (0.3, 0.3 , 0.4)$, $(\kappa_1, \kappa_2, \kappa_3) = (60,70,80)$ and

		$$a_1 = \frac{1}{3}\begin{pmatrix} \sqrt{3} \\ \sqrt{2} \\ 2 \end{pmatrix},  \quad a_2 = \begin{pmatrix}
		0 \\-1 \\0
		\end{pmatrix}, \quad a_3 = \frac{-\sqrt{2}}{2}\begin{pmatrix}
		1  \\ 0 \\ 1
		\end{pmatrix}.$$
As in the preceding example, the means of the mixture components are selected so that the distances between each component mean and the sample source mean are distinct, obtaining three different components  (see Figure \ref{fig:von}), to construct a $95 \%$ uniform band for the true colocalization. We again consider three sample sizes $n=100, 500, 1000$, and bootstrap replications $B=1000$. The corresponding empirical EOT curves are computed to obtain the bootstrap sample distribution of $\left\|\left(\varphi\left(\mu_{n}^{\ast},\nu_{n}^{\ast}\right)-\varphi\left(\mu_{n},\nu_{n}\right)\right)\right\|_{\infty}$. Following the same approach as before, a sample of $\left\|\left(\varphi\left(\mu_{n}^{\ast},\nu_{n}^{\ast}\right)-\varphi\left(\mu_{n},\nu_{n}\right)\right)\right\|_{\infty}$, where the limit colocalization curve $\varphi(\mu,\nu)$ is approximated by a sample of size $20 \, 000$, is obtained. Then we compare both samples as we show in Figure \ref{fig:vonMF_qq}. The bootstrap quantile values are  presented in Table \ref{table:quantil_von}. For the confidence bands, see Figure \ref{fig:vonMF_confbands}.

\begin{table}[h!]
	  			\begin{center}
			\begin{tabular}{ |c|c|c|c| } 
 			\hline
 			$n$ & 100 & 500 & 1000 \\
 			\hline
			$\widehat{q}_{95}^{\ast}$ & 0.9299 & 0.8771 & 0.8638 \\
			\hline 
			$\widehat{q}_{95}^{\ast} \sqrt{\frac{2}{n}}$ & 0.1315 & 0.0555 & 0.0386 \\
		 	\hline
			\end{tabular}
		\end{center}
	\caption{Bootstrap quantiles for von Mises vs von Mises Mixture (scenario II).} 
	\label{table:quantil_von}	
	\end{table}

		    	\begin{figure}[!htbp]
 	 	\centering
 	 	\includegraphics[width=0.8\textwidth]{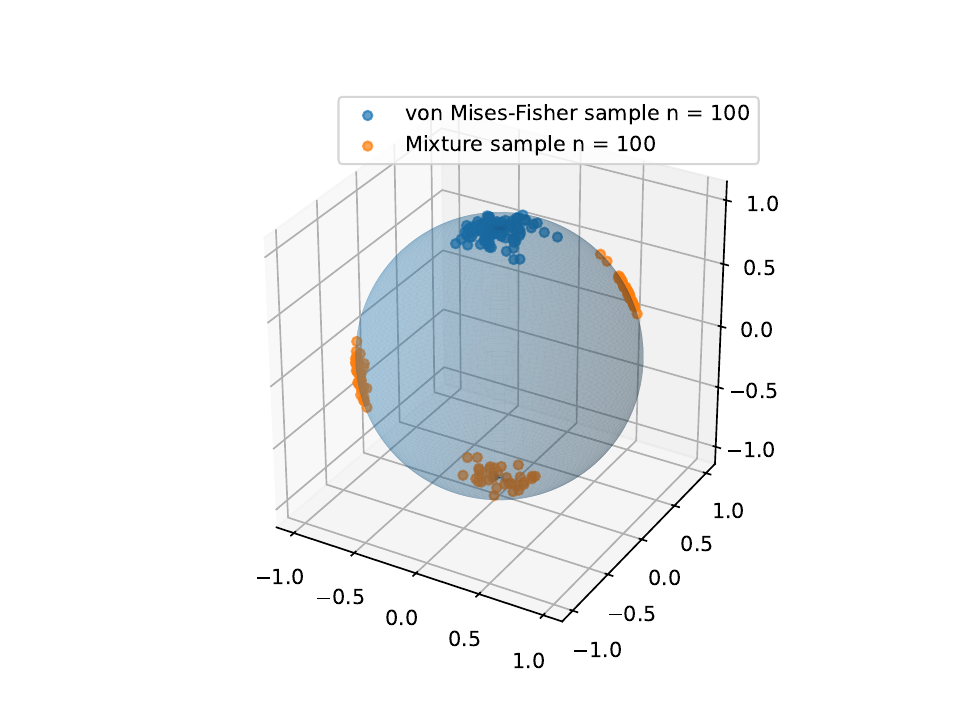}
  		\caption{\textit{Point-cloud, von Mises-Fisher vs Mixture:}  In blue the source sample from the von Misses-Fisher distribution and in orange the target sample from the von Mises mixture. Observe the three components of the mixture, each centered at a different distance from the mean of the source sample distribution.}
 		\label{fig:von}
		\end{figure}		
		
		\begin{figure}[!htbp]
 	 	\centering
  			\includegraphics[width=\textwidth]{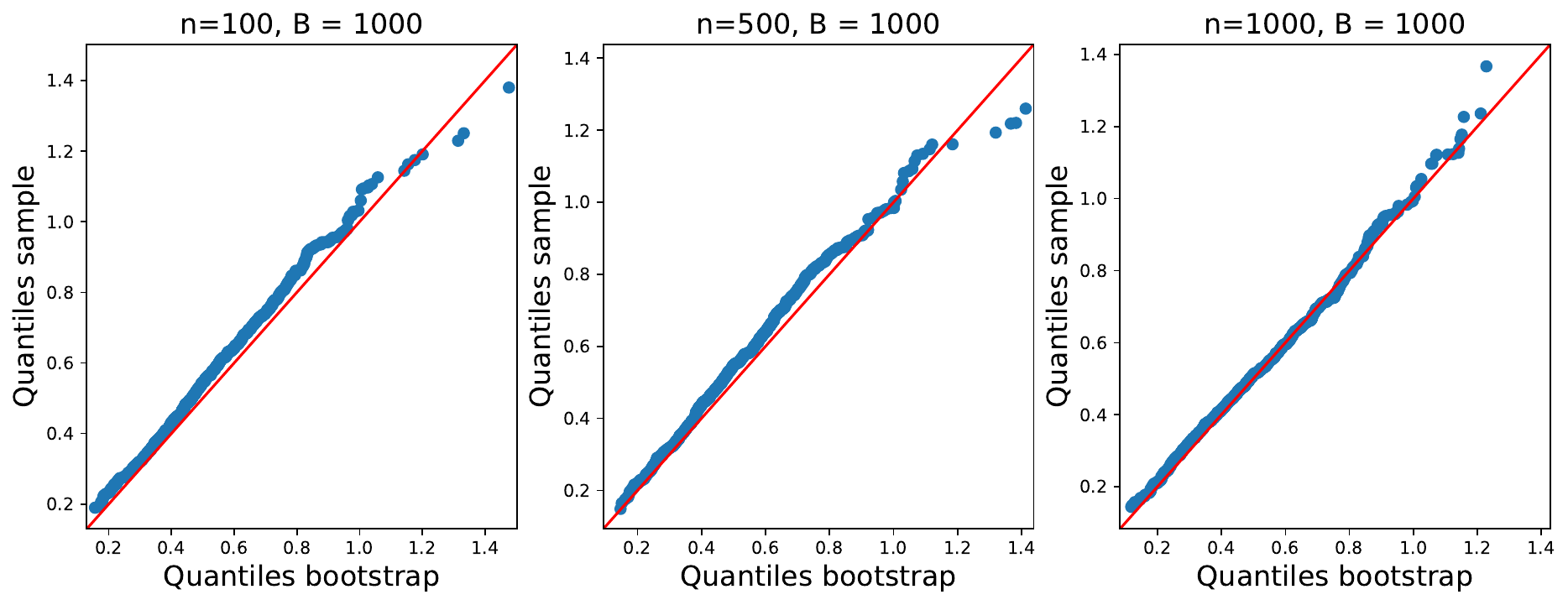}
  		\caption{\textit{Bootstrap accuracy of the EOT limit law for scenario II:} Q-Q plots of the bootstrap distribution of the uniform norm of the EOT for three sample sizes $n=100, 500, 1000$ and bootstrap replications $B=1000$. The underlying true EOT curve is approximated by a sample of size $20\,000$.}
  		\label{fig:vonMF_qq}
  		\end{figure}

\begin{figure}[h!]
			\centering			
			\includegraphics[width=\textwidth]{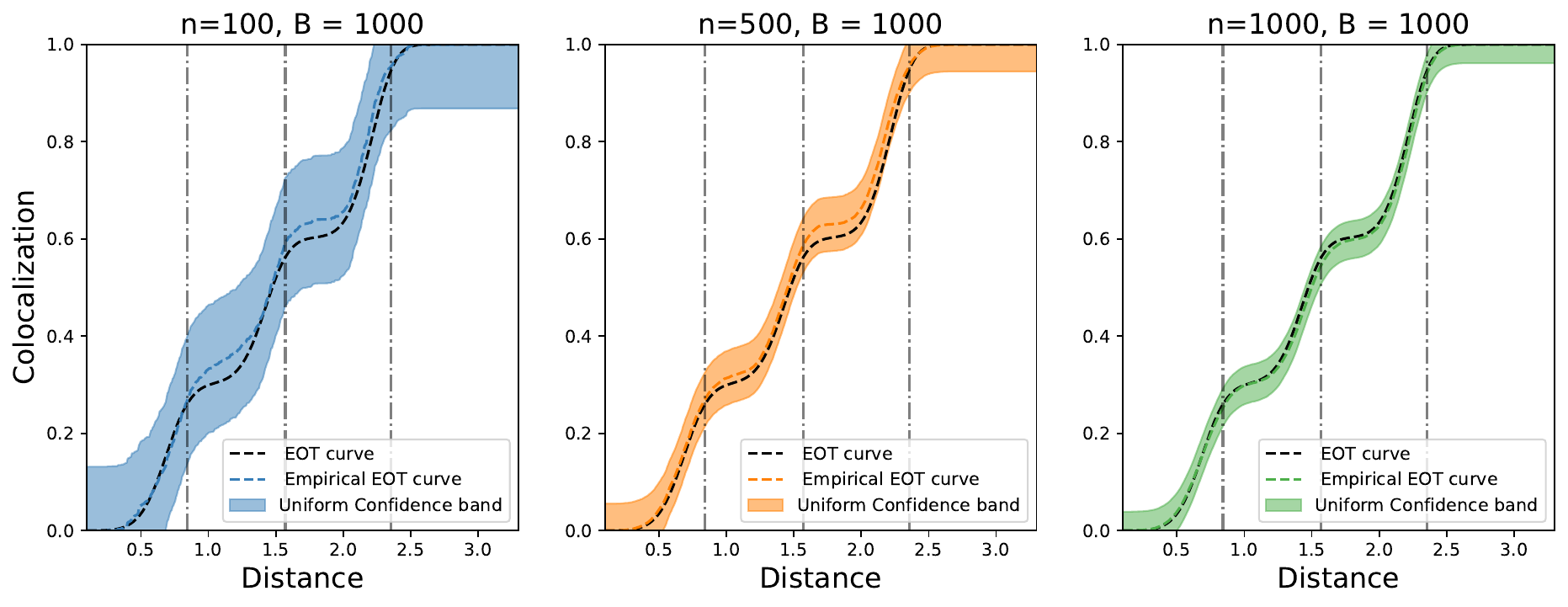}
			\caption{\textit{EOT-confidence bands for scenario II:} the dashed line correspond to the EOT curve (approximation of the ground truth) between the von Mises-Fisher and the von Mises-Fisher mixture distributions (with sample size $20\,000$ and regularization parameter $\lambda = 0.001$), the dotted line correspond to the empirical EOT curve with uniform bootstrap confidence bands at level $1-\alpha$, with $\alpha = 0.05$, based on the $n$ out of $n$ bootstrap with  $B=1000$ replications for $n=100, 500, 1000$. The bands are of the form $\varphi(\mu_n, \nu_n) \pm \widehat{q}_{95}^{\ast}\sqrt{\frac{2}{n}}$ with $\widehat{q}_{95}^{\ast}$ as in the Table \ref{table:quantil_von}. The vertical dashed lines correspond to the distances between the mean of the von Mises-Fisher distribution and the means of the components of the mixture.}
			\label{fig:vonMF_confbands}
		\end{figure}		
In summary, we find that the bootstrap approximation works reasonably well. Further, in both scenarios the empirical regularized colocalization curves clearly capture the three major scales on which mass is being transported to match the three component mixture with the target distribution.	
    	\subsection{Colocalization for super-resolution microscopy}\label{SubSec:RealDataAnalysis}
    	With the advance of super-resolution fluorescence microscopy (nanoscopy), it is possible to acquire detailed information about individual proteins down to the scale of nanometres. Essentially,  all nanoscopies rely on fluorescent markers which are attached to the proteins of interest, called labeling. After laser excitation, these undergo molecular transitions between two or more energy states, usually a fluorescent ``on" and a  non fluorescent ``off" state. The resulting emitted photons in the on-state are then recorded, e.g. by a photon detector, and provide highly resolved spatial images of protein assemblies (for the statistical-physical modeling see e.g. \cite{Munk2020} and the reference therein).  In particular, one of the techniques to achieve super-resolution is STimulated Emission Depletion (STED) super-resolution microscopy, see e.g. \cite{hell2007far}. Meanwhile, this is a well established methodology which has led to significant advances in biological, pharmaceutical and medical research, among others, reflected in the Chemistry Nobel Price, 2014.

In STED nanoscopy the obtained raw image data represent spatial fluorescence distributions (recall Figure \ref{fig:TomMicsection} and see also Figure \ref{fig:TomMic}) that are stored as matrices containing the recorded protein intensities and its coordinates. Based on such images, the understanding of the spatial proximity (colocalization) of protein assemblies (e.g. \cite{adler2010quantifying,zinchuk2014quantitative}) is fundamental for understanding of protein formation, interaction and communication processes in the cell. A variety of methods to analyse colocalization is available (\cite{wang2019spatially, mukherjee2020generalizing, lagache2018mapping}). More recently, the use of colocalization curves in \eqref{Eqn:ClassicalColocFn} has been introduced by \cite{Klatt2020} as a tool which in addition to existing methods provides a concrete matching of proteins. Meanwhile, such  colocalization curves have been employed for the analysis of mitochondria (\cite{Tameling2021,  naas2024multimatch}; \cite{ Hirtl_2025}), blood cells (\cite{Vaisey_2022}), and for studies of neuronal cells in both injured and non-injured retinas (\cite{Strong_2023}). 
		
	To cast this into our framework, in the following we consider the centers $h_i$ of the $L$ pixels $i$ as the ground space $\mathcal{A} =\{h_1, \ldots, h_L \}$, for images of $L = L_x \cdot L_y $ pixels, where $L_x, L_y$, are the number of pixels in the $x-$ and $y-$ direction respectively. Then the recorded fluorophore intensity (after standardization) can be viewed as a discrete probability distribution over the simplex 
    	$$ \Delta_L = \left\{ r = (r_1, \ldots, r_L) \in \mathbb{R}^L : r_i \geq 0, \sum_{i=1}^L r_i = 1 \right\}$$
		on the equidistant grid $[0, L_x] \times [0, L_y ]$. The cost to transport the intensity from one pixel to another one is given by the euclidean distance. Then, for samples of the intensity distributions, $\widehat{\mu}_{L,n} = \sum_{i=1}^L U_{i,L}^\mu \delta_{h_i}$, $\widehat{\nu}_{L,n} = \sum_{i=1}^L U_{i,L}^\nu \delta_{h_i}$, where $U_{i,L}^\mu$ and $U_{i,L}^\nu$ correspond to the intensity recorded in the pixel $i$ for each image, the EOT curve based on the estimated EOT plan $\widehat{\pi}_{\widehat{\mu}_{L,n},\widehat{\nu}_{L,n}}^{\lambda}$ is given by
    	\begin{equation*}
    	\varphi(\widehat{\mu}_{L,n},\widehat{\nu}_{L,n})(t)= \sum_{i,j = 1}^L \mathbf{1}_{\{c\leq t\}}(h_i,h_j) \widehat{\pi}_{\widehat{\mu}_{L,n},\widehat{\nu}_{L,n}}^{\lambda}(h_i,h_j), \qquad\text{for } t \in  [0,\|c\|_{\infty}],
    	\end{equation*}
    	where $\|c\|_{\infty}$ is the maximal distance between points on the ground space $\mathcal{A}$. We mention that this is a further discretization of the underlying empirical process, where the data correspond to the observed intensities in each of the pixels forming the fixed partition of $[0, L_x] \times [0, L_y]$. This means that the data do not constitute a sample over the region but rather represent an intensity function over the partition determined by the image resolution. Further details regarding the asymptotic behavior for this specific case have been moved to Subsection \ref{SuplMat:Consistency} in the Supplementary Material.
    	
In Figure \eqref{fig:TomMic}, a STED image of $L_c = 472 \times 1208 = 570,176$ pixels, and a section of $L_s = 128 \times 128 = 16,384$ pixels, recorded on immunolabeled human cells is displayed. Here, the cells are labelled for a protein in the mitochondrial outer membrane (Tom20) and in the mitochondrial inner membrane (Mic60), respectively. The data is taken from the  Optimal Transport Colocalization  (OTC) dataset \cite{carla_tameling_2021}. Specifically, the images ``07\_00\_Mic60\_STED.tif'' and ``07\_00\_Tom20\_STED.tif'', and the corresponding sections ``07\_03\_Mic60\_STED.tif'' and ``07\_03\_Tom20\_STED.tif'', were used. In this case, the cost matrices are of order $L_c^2 = 3.25 \times 10^{11}$ and $L_s^2=2.68 \times  10^8$. For these complete images, the cost matrix is too large to allow exact computation of the EOT value and plan for small regularization parameters. To overcome this obstacle, we apply our methodology to resampled colocalization curves, thus significantly reducing computational time and complexity. The corresponding bootstrap confidence bands will provide the coverage of the exact colocalization curve with high probability. 

In the following, we analyze runtime, the influence of the regularization parameter and the coverage property of the proposed uniform confidence bands. For all computations we use the \verb!Julia! package \emph{MuSink} (\cite{Musink2025Staudt}) with \verb!Julia! version 1.11.6.
   	   
		\subsubsection*{Influence of the regularization parameter $\lambda$.} For larger values of the regularization parameter $\lambda$ in \eqref{Eqn:EOT}, the entropic term dominates and the transport plan becomes blurred, approaching to the independent coupling. Conversely, as $\lambda$ tends to zero, the plan converges to the (non-penalized) optimal plan (\cite{Nutz2021}) at the expense of increasing numerical instability of the Sinkhorn's algorithm (see the references given in the Introduction). To investigate this influence further, we consider the colocalization curve,  $\varphi(\widehat{\mu}_{L_s,n}, \widehat{\nu}_{L_s,n})$, between the STED recordings of Tom20 and Mic60 (see Figure \eqref{fig:TomMic}), for different values of regularization parameter $\lambda = 0.1, 0.01, 0.001$ and the corresponding colocalization curve from the (non-penalized) optimal plan which has been exactly computed (with large numerical effort using the commercial version of the CPLEX solver) as in \cite{Tameling2021}. For the bootstrap confidence band, we sample $n=2000$ times according to the intensity distribution to obtain $\widehat{\mu}_{L_s,n}^*$ and $\widehat{\nu}_{L_s,n}^*$ for each colocalization curve and each regularization parameter. These samples are then used to compute the colocalization curve  $\varphi(\widehat{\mu}_{L_s,n}^*,\widehat{\nu}_{L_s,n}^*)$. A $95\%$ bootstrap confidence band with $B=2000$ repetitions is constructed as it is described in Section \ref{Sec:StatApplications}. The quantile values are reported in Table \eqref{table:quantil_STEDcomplete}, and the corresponding confidence bands are shown in Figure \eqref{fig:TomMic_diff_lambda}, together with the exact computation for the non-penalized OT $(\lambda = 0)$, using the  dataset in \cite{carla_tameling_2021}.  We observe that as the regularization parameter decreases, the EOT curve approaches the non-penalized colocalization curve. Notably, the $95\%$ band even for $\lambda=0.001$ covers the unregularized colocalization curve.

		\begin{table}[h!]
	  			\begin{center}
			\begin{tabular}{ |c|c|c|c| } 
 			\hline
 			$\lambda$ & 0.001 & 0.01  & 0.1 \\
 			\hline
 			$\widehat{q}_{95}^{\ast}$ & 2.3836 & 1.3914 & 0.3375 \\ 
 			\hline
			$\widehat{q}_{95}^{\ast}\sqrt{\frac{2}{n}}$ & 0.0754 & 0.0440 & 0.0107 \\
		 	\hline
			\end{tabular}
		\end{center}
	\caption{Bootstrap quantiles for Tom20/Mic60 section for different regularization parameters $\lambda$.} 
	\label{table:quantil_STEDcomplete}	
	\end{table}
		
	\begin{figure}[h!]
			\centering			
			\includegraphics[width=0.8\textwidth]{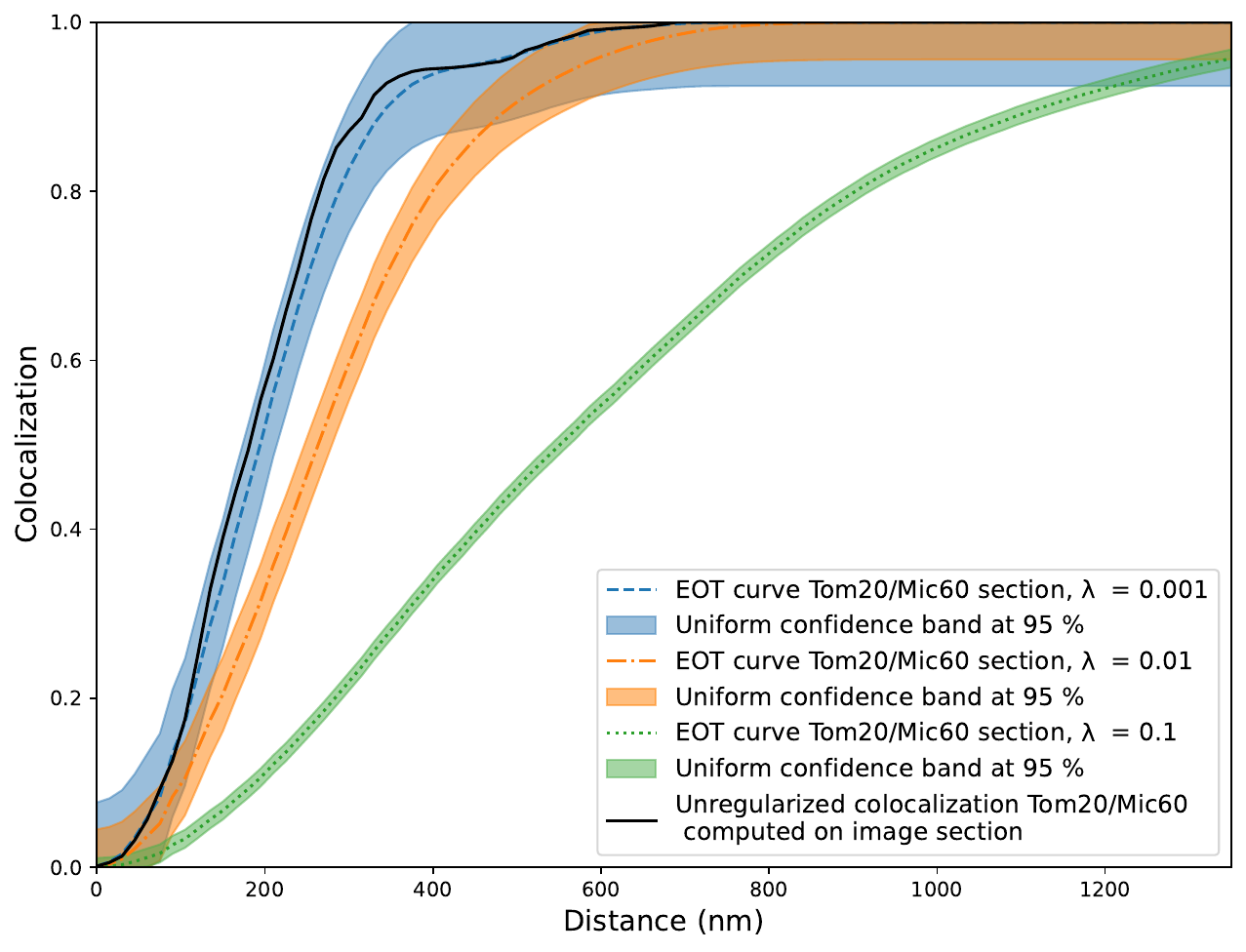}
			\caption{\textit{Influence of the regularization parameter $\lambda$.} The dashed lines correspond to the EOT curves between Tom20, Mic60 computed from the image section of size $128 \times 128$ (see Figure \ref{fig:TomMic}) with different regularization parameters $\lambda = 0.001, 0.01, 0.1$, with uniform bootstrap confidence bands at level $1-\alpha$, for $\alpha= 0.05$, based on the $n$ out of $n$ bootstrap with $B=2000$. The solid line correspond to the colocalization curve from the non-penalized optimal transport plan.}
			\label{fig:TomMic_diff_lambda}
			\end{figure}

		\subsubsection*{Coverage property of subsampling.} We aim to validate whether the colocalization curve is covered by the bootstrap confidence band based on smaller subsamples and for different regularization parameters. Specifically, for a fixed significance level $\alpha$, we determine how closely the empirical coverage probability approximates the nominal coverage $1-\alpha$. To this end, we consider the complete STED images Tom20 and Mic60, Figure \eqref{fig:TomMic}. In this case, the cost matrix is of order $3.25 \times 10^{11}$, making exact computation of the unregularized plan unfeasible. In our numerical experiments, we were able to decrease the  regularization parameter to $\lambda = 0.0008$ and compute the corresponding EOT plan and the colocalization curve
$\varphi(\widehat{\mu}_{L_c,n},\widehat{\nu}_{L_c,n})$. It is worth noticing that computation of the EOT plan requires only a few minutes on a standard laptop. However, evaluating a single colocalization value takes approximately $0.5$ minutes, as it involves matrix multiplication of the cost matrix and transport plan. Since the colocalization curve is obtained by evaluating over a grid of values, the computational cost scales linearly with the grid size. Therefore, computing the full curve may require several hours for sufficiently fine grids, and we expected severe reduction of run time when we randomly subsample from the data. 

To investigate its accuracy, we examine the coverage property of our uniform bands based in small subsamples; we sample from the intensity distribution of each image with sizes $n^* = 500, 1000,  2000, 5000$. For these sample sizes, the colocalization curve can be computed for even smaller values of the regularization parameter. For illustration, we consider $\lambda = 0.0008, 0.0001$ (see Figure \ref{fig:TomMicCol}).  In these settings, computation of both the EOT plan and the corresponding colocalization curve requires only a few seconds, which is substantially faster than computation on the complete image. More detailed, we proceeded as follows: for each sample size and each regularization parameter, we construct a bootstrap confidence band, to obtain $\widehat{\mu}_{L_c,n}^*$ and $\widehat{\nu}_{L_c,n}^*$ for each colocalization curve. These samples are then used to compute the colocalization curve $\varphi(\widehat{\mu}_{L_c,n}^*, \widehat{\nu}_{L_c,n}^*)$. A $95 \%$ bootstrap confidence band with $B = 100$ repetitions is constructed as detailed in Section \ref{Sec:StatApplications}. To investigate how well the empirical coverage probability approximates the nominal coverage probability, we repeated the last procedure 100 times. Table \ref{table:coverage} reports the number of confidence bands that fully cover the curve $\varphi(\widehat{\mu}_{L_c,n},\widehat{\nu}_{L_c,n})$ computed from the complete image. The results indicate that the empirical coverage probability is close to the nominal coverage of $0.95$, even for colocalization curves based on small subsamples.

\begin{table}[ht]
\centering
\begin{tabular}{|c|c|c|}
\hline
 & \multicolumn{2}{|c|}{Empirical coverage probability} \\
\cline{2-3}
Subsample $n^*$& $\lambda = 0.0008$ & $\lambda = 0.0001$ \\
\hline
500  & 0.94 & 0.97 \\
1000 & 0.94 & 0.98 \\
2000 & 0.96 & 0.93 \\
5000 & 0.96 & 0.94 \\
\hline
\end{tabular}
\caption{Empirical coverage probability of bootstrap confidence, with different resampling sizes and regularization parameters, with bootstrap replications $B = 100$ and nominal coverage of $1-\alpha = 0.95$.}
\label{table:coverage}
\end{table}

To summarize, we find that the colocalization curve and its confidence band remain stable, even for small values of $\lambda$, which resemble the physical scale of actual transport well in our situation. Finally, we give a heuristic explanation of the computing benefits observed in our empirical study. Given that the computational complexity of the Sinkhorn algorithm is of order $\mathcal{O}(n^2/\lambda^2)$ (\cite{Dvurechensky2018}), computing bootstrap confidence bands introduces an additional factor corresponding to the number of bootstrap iterations. However, the gain of speed induced by subsampling compensates for this and leads to substantially faster computations. In our case, using the complete image with $L_c = 472 \times 1208 $ pixels and regularization $\lambda = 0.0008$ yields a computational cost of order $ L_c^2/\lambda^2  \approx 5.07 \times 10^{17}$. In contrast, the bootstrap procedure with subsampling $(n^* = 5000, B =100)$ results in a computational order of $B(n^*)^2/\lambda^2  \approx 3.9  \times 10^{15}$, corresponding to a reduction of roughly 2 orders of magnitude.

    		\begin{figure}[htb]
			\centering			
			\includegraphics[width=1\textwidth]{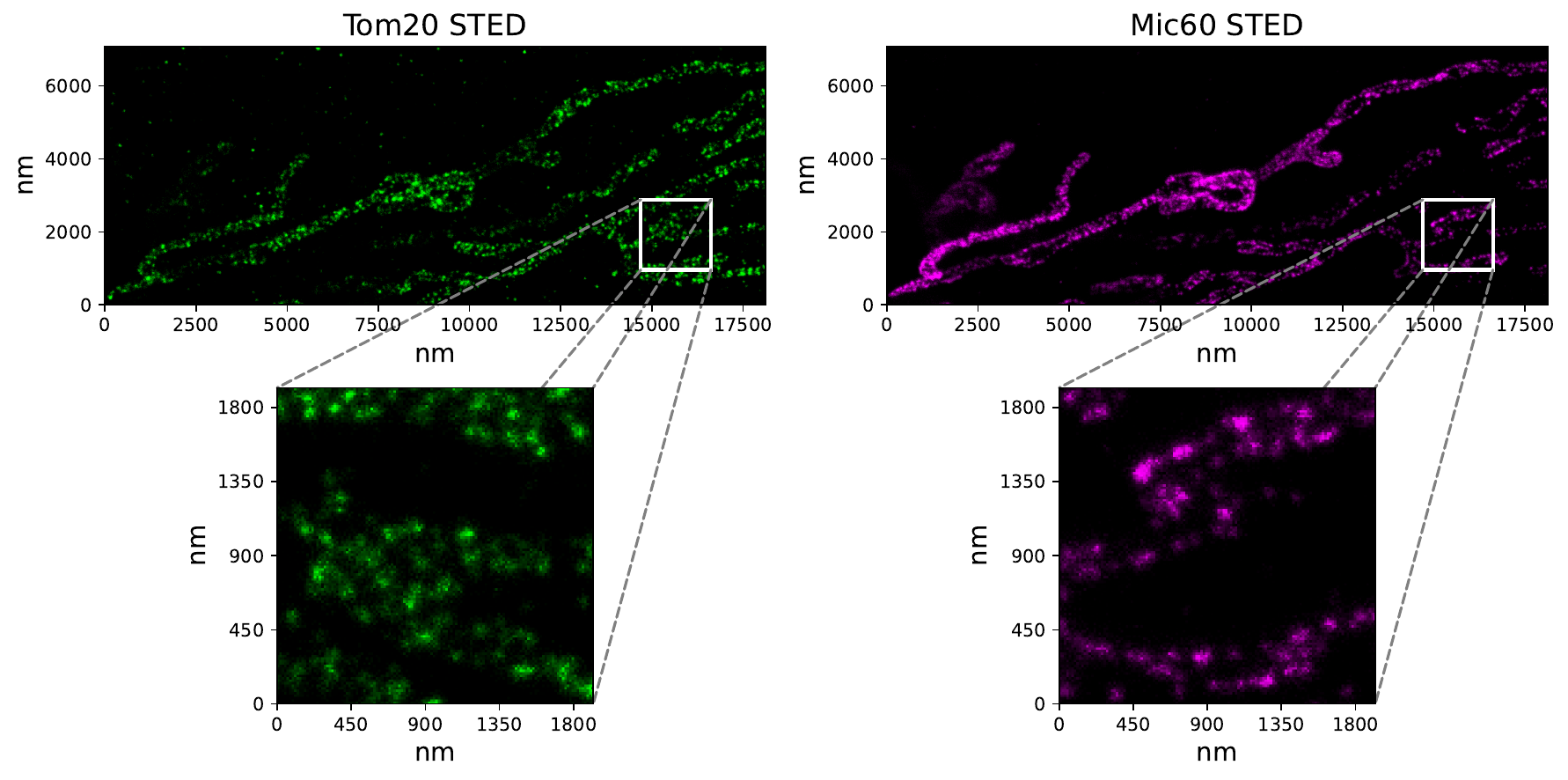}
			\caption{\textit{STED nanoscopy images of Tom20 and Mic60.} Top: Complete image of size $472 \times 1208$ pixels from the Optimal Transport Colocalization (OTC) data set \textcite{carla_tameling_2021}. For visualization, pixel intensities were scaled by factors of $5$ (Tom20) and $2.5$ (Mic60) to enhance brightness and contrast. Bottom: Zoom in ($128 \times 128$ pixels) at the grey rectangles which correspond to the sections in Figures \ref{fig:TomMicsection} and \ref{fig:TomMic_diff_lambda}.}
			\label{fig:TomMic}
			\end{figure}

			\begin{figure}[!h]
			\centering			
			\includegraphics[width=\textwidth]{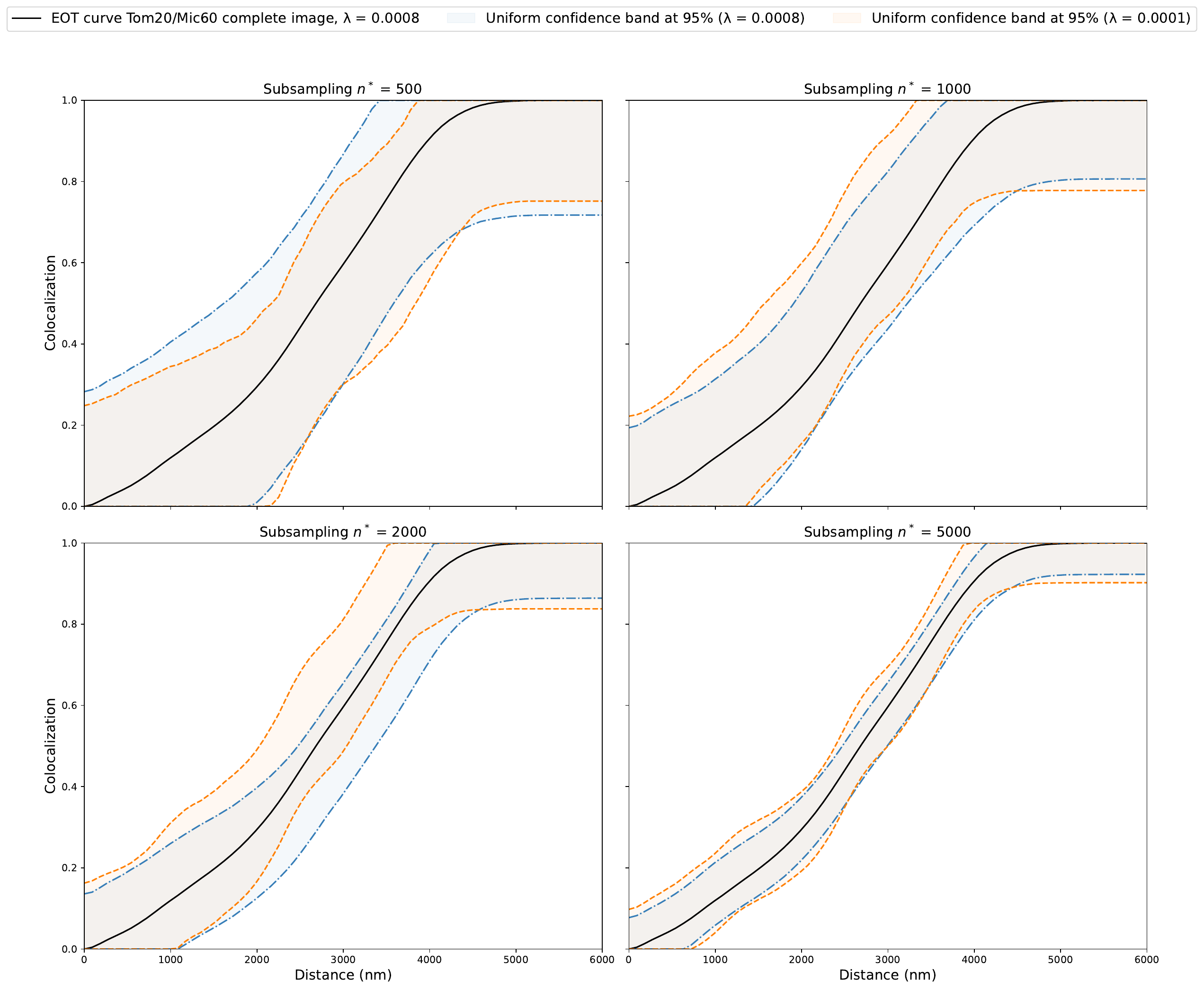}
			\caption{\textit{Colocalization curve.} EOT curve $\varphi(\widehat{\mu}_{L,n}, \widehat{\nu}_{L,n})$ between the full Tom20 and Mic60 images (see Figure \ref{fig:TomMic}) with regularization parameter $\lambda = 0.0008$. Uniform bootstrap confidence bands at level $1-\alpha$ with $\alpha=0.05$ are constructed using the $n^*$ out $n^*$ bootstrap with $B=100$ replications for $n^* = 500, 1000, 2000, 5000$ with different regularization parameters $\lambda = 0.0008, 0.0001$. In both cases, each confidence band fully contains the EOT curve $\varphi(\widehat{\mu}_{L,n}, \widehat{\nu}_{L,n})$.}
			\label{fig:TomMicCol}
			\end{figure}

\paragraph{Acknowledgements} All authors gratefully acknowledge support from the Deutsche Forschungsgemeinschaft through DFG RU 5381. We also thank Stefan Jakobs from the Max Planck Institute for Multidisciplinary Science for providing the STED data from Section \eqref{SubSec:RealDataAnalysis}. Additionally, we thank Christian B\"ohm for technical support and Thomas Staudt for computational advice.
   
\printbibliography

@article{Adler2017,
	title={Learning to solve inverse problems using {W}asserstein loss},
	author={Adler, Jonas and Ringh, Axel and {\"O}ktem, Ozan and Karlsson, Johan},
	journal={arXiv preprint arXiv:1710.10898},
	year={2017}
}

@article{Altschuler2017,
	title={Near-linear time approximation algorithms for optimal transport via Sinkhorn iteration},
	author={Altschuler, Jason and Niles-Weed, Jonathan and Rigollet, Philippe},
	journal={Advances in neural information processing systems},
	volume={30},
	pages={1961--1971},
	year={2017},
	doi={10.5555/3294771.3294958}
}

@article{Altschuler2022,
	author={Altschuler, Jason M. and Niles-Weed, Jonathan and Stromme, Austin J.},
	title={Asymptotics for {S}emidiscrete {E}ntropic {O}ptimal {T}ransport},
	journal={SIAM Journal on Mathematical Analysis},
	volume={54},
	number={2},
	pages={1718-1741},
	year={2022},
	doi={10.1137/21M1440165},
	URL={https://doi.org/10.1137/21M1440165},
}

@inproceedings{Arjovsky2017,
	title={Wasserstein generative adversarial networks},
	author={Arjovsky, Martin and Chintala, Soumith and Bottou, L{\'e}on},
	booktitle={International Conference on Machine Learning},
	pages={214--223},
	year={2017},
	volume={70},
	organization={PMLR}
}

@misc{Balakrishnan2025,
	title={{Statistical Inference for Optimal Transport Maps: Recent Advances and Perspectives}}, 
	author={Sivaraman Balakrishnan and Tudor Manole and Larry Wasserman},
	year={2025},
	eprint={2506.19025},
	archivePrefix={arXiv},
	primaryClass={math.ST},
	url={https://arxiv.org/abs/2506.19025}, 
}

@inproceedings{Balikas2018,
	title={Cross-lingual document retrieval using regularized {W}asserstein distance},
	author={Balikas, Georgios and Laclau, Charlotte and Redko, Ievgen and Amini, Massih-Reza},
	booktitle={Advances in Information Retrieval: 40th European Conference on IR Research, ECIR 2018, Grenoble, France, March 26-29, 2018, Proceedings 40},
	pages={398--410},
	year={2018},
	doi={10.1007/978-3-319-76941-7_30},
	organization={Springer}
}

@article{delBarrio2025,
	title={Distributional Limit Theory for Optimal Transport},
	author={del Barrio, Eustasio and Gonz\'{a}lez-Sanz, Alberto and Loubes, Jean-Michel and Rodr\'{i}guez-V\'{i}tores, David},
	journal={arXiv preprint arXiv:2505.19104},
	year={2025}
}

@book{Billingsley2013,
	title={Convergence of {P}robability {M}easures},
	author={Billingsley, Patrick},
	year={2013},
	publisher={John Wiley \& Sons}
}

@article{Carcamo2020,
	title={Directional differentiability for supremum-type functionals: Statistical applications},
	author={C{\'a}rcamo, Javier and Cuevas, Antonio and Rodr{\'i}guez, Luis-Alberto},
	journal={Bernoulli},
	volume={26},
	number={3},
	pages={2143--2175},
	year={2020},
	doi={10.3150/19-BEJ1188}
}

@article{Carlier&Laborde2020,
	title={A differential approach to the multi-marginal {S}chrödinger system},
	author={Carlier, Guillaume and Laborde, Maxime},
	journal={SIAM Journal on Mathematical Analysis},
	volume={52},
	number={1},
	pages={709--717},
	year={2020},
	doi={10.1137/19M1253800},
	publisher={SIAM}
}

@inproceedings{Chapel2021,
	author={Chapel, {L}aetitia and {F}lamary, {R}\'{e}mi and {W}u, {H}aoran and {F}\'{e}votte, {C}\'{e}dric and {G}asso, {G}illes},
	booktitle={Advances in {N}eural {I}nformation {P}rocessing {S}ystems},
	editor={M. {R}anzato and A. {B}eygelzimer and Y. {D}auphin and P.S. {L}iang and J. {W}ortman {V}aughan},
	pages={23270--23282},
	publisher={Curran {A}ssociates, {I}nc.},
	title={Unbalanced {O}ptimal {T}ransport through {N}on-negative {P}enalized {L}inear {R}egression},
	url={https://proceedings.neurips.cc/paper_files/paper/2021/file/c3c617a9b80b3ae1ebd868b0017cc349-Paper.pdf},
	volume={34},
	year={2021}
}

@book{Chewi2025,
	title={Statistical optimal transport},
	author={Chewi, Sinho and Niles-Weed, Jonathan and Rigollet, Philippe},
	year={2025},
	publisher={Springer}
}

@article{Cuturi2013,
	title={Sinkhorn distances: Lightspeed computation of optimal transport},
	author={Cuturi, Marco},
	journal={Advances in {N}eural {I}nformation {P}rocessing {S}ystems},
	volume={26},
	year={2013},
	doi={10.5555/2999792.2999868}
}

@article{Cuturi&Peyre2019,
	url = {http://dx.doi.org/10.1561/2200000073},
	year = {2019},
	volume = {11},
	journal = {Foundations and Trends® in Machine Learning},
	title = {Computational Optimal Transport: With Applications to Data Science},
	doi = {10.1561/2200000073},
	issn = {1935-8237},
	number = {5-6},
	pages = {355-607},
	author = {Marco Cuturi and Gabriel Peyr\'{e}}
}

@article{Dette&Kokot2022,
	title={Detecting relevant differences in the covariance operators of functional time series: a sup-norm approach},
	author={Dette, Holger and Kokot, Kevin},
	journal={Annals of the Institute of Statistical Mathematics},
	volume={74},
	number={2},
	pages={195--231},
	year={2022},
	doi={10.1007/s10463-021-00795-2},
	publisher={Springer}
}

@article{Dumbgen1993,
	title={On nondifferentiable functions and the bootstrap},
	author={D{\"u}mbgen, Lutz},
	journal={Probability Theory and Related Fields},
	volume={95},
	pages={125--140},
	year={1993},
	doi={10.1007/BF01197342},
	publisher={Springer}
}

@InProceedings{Dvurechensky2018,
	title={Computational {O}ptimal {T}ransport: {C}omplexity by {A}ccelerated {G}radient {D}escent {I}s {B}etter {T}han by {S}inkhorn's {A}lgorithm},
	author={Dvurechensky, Pavel and Gasnikov, Alexander and Kroshnin, Alexey},
	booktitle={Proceedings of the 35th {I}nternational {C}onference on {M}achine {L}earning},
	pages={1367--1376},
	year={2018},
	editor={Dy, {J}ennifer and {K}rause, {A}ndreas},
	volume={80},
	series={Proceedings of Machine Learning Research},
	month={10--15 Jul},
	publisher={PMLR},
	url={https://proceedings.mlr.press/v80/dvurechensky18a.html}
}

@article{Ferradans2014,
	title={Regularized discrete optimal transport},
	author={Ferradans, Sira and Papadakis, Nicolas and Peyr{\'e}, Gabriel and Aujol, Jean-Fran{\c{c}}ois},
	journal={SIAM Journal on Imaging Sciences},
	volume={7},
	number={3},
	pages={1853--1882},
	year={2014},
	doi={10.1137/130929886},
	publisher={SIAM}
}

@book{Folland2013,
	title={Real {A}nalysis: {M}odern {T}echniques and their {A}pplications},
	author={Folland, Gerald B},
	volume={40},
	year={1999},
	publisher={John Wiley \& Sons}
}

@book{Galichon2018,
	title={Optimal {T}ransport {M}ethods in {E}conomics},
	author={Galichon, Alfred},
	year={2018},
	doi={10.1515/9781400883592},
	publisher={Princeton University Press}
}

@article{Genevay2016,
	title={Stochastic optimization for large-scale optimal transport},
	author={Genevay, Aude and Cuturi, Marco and Peyr{\'e}, Gabriel and Bach, Francis},
	journal={Advances in neural information processing systems},
	volume={29},
	doi={10.5555/3157382.3157482},
	year={2016}
}

@inproceedings{Genevay2019,
	title={Sample complexity of {S}inkhorn divergences},
	author={Genevay, Aude and Chizat, L{\'e}naic and Bach, Francis and Cuturi, Marco and Peyr{\'e}, Gabriel},
	booktitle={The 22nd {I}nternational {C}onference on {A}rtificial {I}ntelligence and {S}tatistics},
	pages={1574--1583},
	year={2019},
	organization={PMLR}
}

@article{Goldfeld2024,
	title={Limit theorems for entropic optimal transport maps and Sinkhorn divergence},
	author={Goldfeld, Ziv and Kato, Kengo and Rioux, Gabriel and Sadhu, Ritwik},
	journal={Electronic Journal of Statistics},
	volume={18},
	number={1},
	pages={980--1041},
	year={2024},
	doi={10.1214/24-EJS2217},
	publisher={The Institute of Mathematical Statistics and the Bernoulli Society}
}

@article{Gonzalez2022,
	title={Weak limits of entropy regularized optimal transport; potentials, plans and divergences},
	author={Gonzalez-Sanz, Alberto and Loubes, Jean-Michel and Niles-Weed, Jonathan},
	journal={arXiv preprint arXiv:2207.07427},
	year={2022}
}

@article{Gonzalez&Hundrieser2023,
	title={Weak limits for empirical entropic optimal transport: Beyond smooth costs},
	author={Gonz{\'a}lez-Sanz, Alberto and Hundrieser, Shayan},
	journal={arXiv preprint arXiv:2305.09745},
	year={2023}
}

@article{Groppe&Hundrieser2024,
	title={Lower complexity adaptation for empirical entropic optimal transport},
	author={Groppe, Michel and Hundrieser, Shayan},
	journal={Journal of Machine Learning Research},
	volume={25},
	number={344},
	pages={1--55},
	year={2024}
}

@inproceedings{Kantorovich1942,
	title={On the translocation of masses},
	author={Kantorovich, Leonid V},
	booktitle={Dokl. Akad. Nauk. USSR (NS)},
	volume={37},
	pages={199--201},
	year={1942},
	doi={10.1007/s10958-006-0049-2}
}

@article{Kassraie2024,
	title={Progressive entropic optimal transport solvers},
	author={Kassraie, Parnian and Pooladian, Aram-Alexandre and Klein, Michal and Thornton, James and Niles-Weed, Jonathan and Cuturi, Marco},
	journal={Advances in Neural Information Processing Systems},
	volume={37},
	pages={19561--19590},
	year={2024}
}

@article{Klatt2020,
	title={Empirical regularized optimal transport: Statistical theory and applications},
	author={Klatt, Marcel and Tameling, Carla and Munk, Axel},
	journal={SIAM Journal on Mathematics of Data Science},
	volume={2},
	number={2},
	pages={419--443},
	year={2020},
	doi={10.1137/19M1278788},
	publisher={SIAM}
}

@misc{Liu2025,
	title={Beyond entropic regularization: Debiased Gaussian estimators for discrete optimal transport and general linear programs}, 
	author={Shuyu Liu and Florentina Bunea and Jonathan Niles-Weed},
	year={2025},
	eprint={2505.04312},
	archivePrefix={arXiv},
	primaryClass={math.ST},
	url={https://arxiv.org/abs/2505.04312}, 
}

@article{Manole2024,
	title={Plugin estimation of smooth optimal transport maps},
	author={Manole, Tudor and Balakrishnan, Sivaraman and Niles-Weed, Jonathan and Wasserman, Larry},
	journal={The Annals of Statistics},
	volume={52},
	number={3},
	pages={966--998},
	year={2024},
	publisher={Institute of Mathematical Statistics}
}

@article{Monge1781,
	title={M{\'e}moire sur la th{\'e}orie des d{\'e}blais et des remblais},
	author={Monge, Gaspard},
	journal={Mem. Math. Phys. Acad. Royale Sci.},
	pages={666--704},
	year={1781}
}

@article{Nutz2021,
	title={Introduction to entropic optimal transport},
	author={Nutz, Marcel},
	journal={Lecture notes, Columbia University},
	year={2021},
	url={https://www.math.columbia.edu/~mnutz/docs/EOT_lecture_notes.pdf}
}

@article{Nutz&Wiesel2022,
	title={Entropic optimal transport: {C}onvergence of potentials},
	author={Nutz, Marcel and Wiesel, Johannes},
	journal={Probability Theory and Related Fields},
	volume={184},
	number={1},
	pages={401--424},
	year={2022},
	publisher={Springer},
	doi={10.1007/s00440-021-01096-8}
}

@article{Pal2024,
	title={On the difference between entropic cost and the optimal transport cost},
	author={Pal, Soumik},
	journal={The Annals of Applied Probability},
	volume={34},
	number={1B},
	pages={1003--1028},
	year={2024},
	doi={10.1214/23-AAP1983},
	publisher={Institute of Mathematical Statistics}
}

@book{Rachev&Rueschendorf1998,
	title={Mass Transportation Problems: Volume 1: Theory},
	author={Svetlozar T. Rachev and Ludger R\"uschendorf},
	series={Probability and Its Applications},
	year={1998},
	publisher={Springer New York, NY},
	doi={10.1007/b98893},
	isbn={978-0-387-98350-9}
}

@article{Rigollet&Stromme2025,
	title={On the sample complexity of entropic optimal transport},
	author={Rigollet, Philippe and Stromme, Austin J},
	journal={The Annals of Statistics},
	volume={53},
	number={1},
	pages={61--90},
	year={2025},
	publisher={Institute of Mathematical Statistics}
}

@article{Romisch2004,
	title={Delta Method, Infinite Dimensional},
	author={R{\"o}misch, Werner},
	journal={Encyclopedia of Statistical Sciences},
	volume={3},
	year={2004},
	doi={10.1002/9781118445112.stat01949},
	publisher={Wiley Online Library}
}

@article{Sadhu2025,
	title={Approximation rates of entropic maps in semidiscrete optimal transport},
	author={Sadhu, Ritwik and Goldfeld, Ziv and Kato, Kengo},
	journal={Electronic Communications in Probability},
	volume={30},
	pages={1--13},
	year={2025},
	doi={10.1214/25-ECP682},
	publisher={The Institute of Mathematical Statistics and the Bernoulli Society}
}

@book{Santambrogio2015,
	title={Optimal {T}ransport for {A}pplied {M}athematicians},
	author={Santambrogio, Filippo},
	volume={87},
	year={2015},
	doi={10.1007/978-3-319-20828-2},
	publisher={Springer}
}

@article{Schiebinger2019,
	title={Optimal-transport analysis of single-cell gene expression identifies developmental trajectories in reprogramming},
	author={Schiebinger, Geoffrey and Shu, Jian and Tabaka, Marcin and Cleary, Brian and Subramanian, Vidya and Solomon, Aryeh and Gould, Joshua and Liu, Siyan and Lin, Stacie and Berube, Peter and others},
	journal={Cell},
	volume={176},
	number={4},
	pages={928--943},
	year={2019},
	doi={10.1016/j.cell.2019.01.006},
	publisher={Elsevier}
}

@article{Schmitz2018,
	title={Wasserstein dictionary learning: Optimal transport-based unsupervised nonlinear dictionary learning},
	author={Schmitz, Morgan A and Heitz, Matthieu and Bonneel, Nicolas and Ngole, Fred and Coeurjolly, David and Cuturi, Marco and Peyr{\'e}, Gabriel and Starck, Jean-Luc},
	journal={SIAM Journal on Imaging Sciences},
	volume={11},
	number={1},
	pages={643--678},
	year={2018},
	doi={10.1137/17M1140431},
	publisher={SIAM}
}

@article{Schrieber2016,
	title={Dotmark--a benchmark for discrete optimal transport},
	author={Schrieber, J{\"o}rn and Schuhmacher, Dominic and Gottschlich, Carsten},
	journal={IEEE Access},
	volume={5},
	pages={271--282},
	year={2016},
	doi={10.1109/ACCESS.2016.2639065},
	publisher={IEEE}
}

@article{Shapiro1990,
	title={On concepts of directional differentiability},
	author={Shapiro, Alexander},
	journal={Journal of {O}ptimization {T}heory and {A}pplications},
	volume={66},
	pages={477--487},
	year={1990},
	doi={10.1007/BF00940933},
	publisher={Citeseer}
}

@article{Shapiro1991,
	title={Asymptotic analysis of stochastic programs},
	author={Shapiro, Alexander},
	journal={Annals of Operations Research},
	volume={30},
	pages={169--186},
	year={1991},
	doi={10.1007/BF02204815},
	publisher={Springer}
}

@article{Sommerfeld2019,
	title={Optimal transport: Fast probabilistic approximation with exact solvers},
	author={Sommerfeld, Max and Schrieber, J{\"o}rn and Zemel, Yoav and Munk, Axel},
	journal={Journal of Machine Learning Research},
	volume={20},
	number={105},
	pages={1--23},
	year={2019},
	doi={10.1145/3287324.3287485}
}

@article{NaturePhotonicsSuper-resolution2025,
	title   = {Super-resolution microscopy at its sharpest},
	journal = {Nature Photonics},
	year    = {2025},
	volume  = {19},
	pages   = {219},
	doi     = {10.1038/s41566-025-01632-1}
}

@article{Tameling2021,
	title={Colocalization for super-resolution microscopy via optimal transport},
	author={Tameling, Carla and Stoldt, Stefan and Stephan, Till and Naas, Julia and Jakobs, Stefan and Munk, Axel},
	journal={Nature {C}omputational {S}cience},
	volume={1},
	number={3},
	pages={199--211},
	year={2021},
	doi={10.1038/s43588-021-00050-x},
	publisher={Nature Publishing Group US New York}
}

@book{Vanderbei2020,
	title={Linear Programming: Foundations and Extensions},
	author={Vanderbei, Robert J},
	volume={285},
	year={2020},
	doi={10.1007/978-3-030-39415-8},
	publisher={Springer Nature}
}

@book{van_der_Vaart2000,
	title={Asymptotic {S}tatistics},
	author={van der Vaart, Aad W},
	volume={3},
	year={2000},
	doi={10.1017/CBO9780511802256},
	publisher={Cambridge university press}
}

@book{van_der_Vaart&Wellner2023,
	title={Weak {C}onvergence and {E}mpirical {P}rocesses: {W}ith {A}pplications to {S}tatistics},
	author={van der Vaart, A. W. and Wellner, J. A.},
	isbn={9783031290404},
	series={Springer Series in Statistics},
	%	url={https://books.google.es/books?id=vfzKEAAAQBAJ},
	year={2023},
	doi={10.1007/978-1-4757-2545-2},
	edition={2},
	publisher={Springer International Publishing}
}

@book{Villani2008,
	title={Optimal {T}ransport: {O}ld and {N}ew},
	author={Villani, C.},
	isbn={9783540710509},
	lccn={2008932183},
	series={Grundlehren der mathematischen Wissenschaften},
	%	url={https://books.google.es/books?id=hV8o5R7\_5tkC},
	year={2008},
	doi={10.1007/978-3-540-71050-9},
	publisher={Springer Berlin, Heidelberg}
}

@book{Zeidler2012,
	title={Applied {F}unctional {A}nalysis: {M}ain {P}rinciples and {T}heir {A}pplications},
	author={Zeidler, E.},
	isbn={9781461208211},
	lccn={94043219},
	series={Applied {M}athematical {S}ciences},
	url={https://books.google.de/books?id=abvkBwAAQBAJ},
	year={2012},
	publisher={Springer {N}ew {Y}ork}
}

@article{Zolotarev1983,
	title={Probability metrics},
	author={Zolotarev, Vladimir Mikhailovich},
	journal={Teoriya Veroyatnostei i ee Primeneniya},
	volume={28},
	number={2},
	pages={264--287},
	year={1983},
	doi={10.1137/1128025},
	publisher={Russian Academy of Sciences, Steklov Mathematical Institute of Russian~…}
}

@article{flamary2021pot,
  author  = {R{\'e}mi Flamary and Nicolas Courty and Alexandre Gramfort and Mokhtar Z. Alaya and Aur{\'e}lie Boisbunon and Stanislas Chambon and Laetitia Chapel and Adrien Corenflos and Kilian Fatras and Nemo Fournier and L{\'e}o Gautheron and Nathalie T.H. Gayraud and Hicham Janati and Alain Rakotomamonjy and Ievgen Redko and Antoine Rolet and Antony Schutz and Vivien Seguy and Danica J. Sutherland and Romain Tavenard and Alexander Tong and Titouan Vayer},
  title   = {POT: Python Optimal Transport},
  journal = {Journal of Machine Learning Research},
  year    = {2021},
  volume  = {22},
  number  = {78},
  pages   = {1-8},
  url     = {http://jmlr.org/papers/v22/20-451.html}
}

@article{adler2010quantifying, title={Quantifying colocalization by correlation: The {P}earson correlation coefficient is superior to the {M}ander’s overlap coefficient}, volume={77A}, ISSN={1552-4930}, url={http://dx.doi.org/10.1002/cyto.a.20896}, DOI={10.1002/cyto.a.20896}, number={8}, journal={Cytometry Part A}, publisher={Wiley}, author={Adler, Jeremy and Parmryd, Ingela}, year={2010}, month=mar, pages={733–742} }

@article{zinchuk2014quantitative,
  title={Quantitative colocalization analysis of fluorescence microscopy imagest},
  author={Zinchuk, Vadim and Grossenbacher-Zinchuk, Olga},
  journal={Current protocols in cell biology},
  volume={62},
  number={1},
  pages={4.19.1--4.19.4},
  year={2014},
  publisher={Wiley Online Library},
  doi={10.1002/0471143030.cb0419s62}
}

@article{hell2007far, title={Far-Field Optical Nanoscopy}, volume={316}, ISSN={1095-9203}, url={http://dx.doi.org/10.1126/science.1137395}, DOI={10.1126/science.1137395}, number={5828}, journal={Science}, publisher={American Association for the Advancement of Science (AAAS)}, author={Hell, Stefan W.}, year={2007}, month=may, pages={1153–1158} }

@book{mardia2009directional,
  title={Directional {S}tatistics},
  author={Mardia, Kanti V and Jupp, Peter E},
  year={2009},
  publisher={John Wiley \& Sons}
}

@software{carla_tameling_2021,
  author       = {Carla Tameling and
                  Julia Naas},
  title        = {ctameling/OTC: Optimal Transport Colocalization},
  month        = feb,
  year         = 2021,
  publisher    = {Zenodo},
  version      = {v1.0},
  doi          = {10.5281/zenodo.4553632},
  url          = {https://doi.org/10.5281/zenodo.4553632},
}

@article{naas2024multimatch, title={MultiMatch: geometry-informed colocalization in multi-color super-resolution microscopy}, volume={7}, ISSN={2399-3642}, url={http://dx.doi.org/10.1038/s42003-024-06772-8}, DOI={10.1038/s42003-024-06772-8}, number={1}, journal={Communications Biology}, publisher={Springer Science and Business Media LLC}, author={Naas, Julia and Nies, Giacomo and Li, Housen and Stoldt, Stefan and Schmitzer, Bernhard and Jakobs, Stefan and Munk, Axel}, year={2024}, month=sep }

@article{wang2019spatially, title={Spatially Adaptive Colocalization Analysis in Dual-Color Fluorescence Microscopy}, volume={28}, ISSN={1941-0042}, url={http://dx.doi.org/10.1109/tip.2019.2909194}, DOI={10.1109/tip.2019.2909194}, number={9}, journal={IEEE Transactions on Image Processing}, publisher={Institute of Electrical and Electronics Engineers (IEEE)}, author={Wang, Shulei and Arena, Ellen T. and Becker, Jordan T. and Bement, William M. and Sherer, Nathan M. and Eliceiri, Kevin W. and Yuan, Ming}, year={2019}, month=sep, pages={4471–4485} }

@article{mukherjee2020generalizing, title={Generalizing the Statistical Analysis of Objects’ Spatial Coupling in Bioimaging}, volume={27}, ISSN={1558-2361}, url={http://dx.doi.org/10.1109/lsp.2020.3003821}, DOI={10.1109/lsp.2020.3003821}, journal={IEEE Signal Processing Letters}, publisher={Institute of Electrical and Electronics Engineers (IEEE)}, author={Mukherjee, Suvadip and Gonzalez-Gomez, Catalina and Danglot, Lydia and Lagache, Thibault and Olivo-Marin, Jean-Christophe}, year={2020}, pages={1085–1089} }

@article{lagache2018mapping, title={Mapping molecular assemblies with fluorescence microscopy and object-based spatial statistics}, volume={9}, ISSN={2041-1723}, url={http://dx.doi.org/10.1038/s41467-018-03053-x}, DOI={10.1038/s41467-018-03053-x}, number={1}, journal={Nature Communications}, publisher={Springer Science and Business Media LLC}, author={Lagache, Thibault and Grassart, Alexandre and Dallongeville, Stéphane and Faklaris, Orestis and Sauvonnet, Nathalie and Dufour, Alexandre and Danglot, Lydia and Olivo-Marin, Jean-Christophe}, year={2018}, month=feb }

@misc{Musink2025Staudt,
  author = {Thomas Staudt},
  title = {{MuSink}},
  year = {2025},
  url = {https://github.com/tscode/MuSink.jl}}

@article{Hirtl_2025, title={A novel super-resolution STED microscopy analysis approach to observe spatial MCU and MICU1 distribution dynamics in cells}, volume={1872}, ISSN={0167-4889}, url={http://dx.doi.org/10.1016/j.bbamcr.2025.119900}, DOI={10.1016/j.bbamcr.2025.119900}, number={3}, journal={Biochimica et Biophysica Acta (BBA) - Molecular Cell Research}, publisher={Elsevier BV}, author={Hirtl, Martin and Gottschalk, Benjamin and Bachkoenig, Olaf A. and Oflaz, Furkan E. and Madreiter-Sokolowski, Corina and Høydal, Morten Andre and Graier, Wolfgang F.}, year={2025}, month=mar, pages={119900} }

@article{Vaisey_2022, title={Piezo1 as a force-through-membrane sensor in red blood cells}, volume={11}, ISSN={2050-084X}, url={http://dx.doi.org/10.7554/elife.82621}, DOI={10.7554/elife.82621}, journal={eLife}, publisher={eLife Sciences Publications, Ltd}, author={Vaisey, George and Banerjee, Priyam and North, Alison J and Haselwandter, Christoph A and MacKinnon, Roderick}, year={2022}, month=dec }

@article{Strong_2023, title={Activation of multiple Eph receptors on neuronal membranes correlates with the onset of optic neuropathy}, volume={10}, ISSN={2326-0254}, url={http://dx.doi.org/10.1186/s40662-023-00359-w}, DOI={10.1186/s40662-023-00359-w}, number={1}, journal={Eye and Vision}, publisher={Springer Science and Business Media LLC}, author={Strong, Thomas A. and Esquivel, Juan and Wang, Qikai and Ledon, Paul J. and Wang, Hua and Gaidosh, Gabriel and Tse, David and Pelaez, Daniel}, year={2023}, month=oct }

@Inbook{Munk2020,
author="Munk, Axel
and Staudt, Thomas
and Werner, Frank",
editor="Salditt, Tim
and Egner, Alexander
and Luke, D. Russell",
title="Statistical Foundations of Nanoscale Photonic Imaging",
bookTitle="Nanoscale Photonic Imaging",
year="2020",
publisher="Springer International Publishing",
address="Cham",
pages="125--143",
isbn="978-3-030-34413-9",
doi="10.1007/978-3-030-34413-9_4",
url="https://doi.org/10.1007/978-3-030-34413-9_4"
}
\clearpage	

\section{Supplementary material}
	The aim of this section is to provide all the mathematical details omitted in the body of this paper. We will denote the space of continuous functions on $\mathfrak{X}$ by $\mathcal{C}(\mathfrak{X})$. Given $u\in\mathcal{C}(\mathfrak{X})$, the set $\overline{\{x\in\mathfrak{X}:u(x)\neq0\}}$ is called support of $u$ and denoted by $\operatorname{supp}(u)$, where $\overline{A}$ denotes the topological closure of any set $A\subseteq\mathfrak{X}$. We say that $u$ is of compact support, or that it belongs to the space $\mathcal{C}_{\operatorname{c}}(\mathfrak{X})$, if $\operatorname{supp}(u)$ is compact. It is immediate that $\mathcal{C}_{\operatorname{c}}(\mathfrak{X})\subseteq\mathcal{C}_{\operatorname{b}}(\mathfrak{X})$. Another space of continuous functions to be used in this section is $\mathcal{C}_{0}(\mathfrak{X})$, the space of continuous functions that \enquote{vanishes at infinity}. Formally, $u\in\mathcal{C}_{0}(\mathfrak{X})$ if and only if for every $\varepsilon>0$, the set $\{x\in\mathfrak{X}:|u(x)|\geq\varepsilon\}$ is compact. Then, it is clear that $\mathcal{C}_{\operatorname{c}}(\mathfrak{X})\subseteq\mathcal{C}_{0}(\mathfrak{X})$. Analogously to the notation introduced in subsection \ref{Subsec:Theory}: $\mathcal{C}_{0}(\mathfrak{X})_{1}=\mathcal{C}_{0}(\mathfrak{X})\cap\mathcal{C}_{\operatorname{b}}(\mathfrak{X})_{1}$, the unit ball of $\mathcal{C}_{0}(\mathfrak{X})$; for $s>0$, $\mathcal{C}_{0}^{s}(\mathfrak{X})=\mathcal{C}_{0}(\mathfrak{X})\cap\mathcal{C}_{\operatorname{b}}^{s}(\mathfrak{X})$, the space of $s$-H\"{o}lder continuous functions that vanish at infinity; and $\mathcal{C}_{0}^{s}(\mathfrak{X})_{1}=\mathcal{C}_{0}^{s}(\mathfrak{X})\cap\mathcal{C}_{\operatorname{b}}(\mathfrak{X})_{1}$.
		\subsubsection*{Mathematical approach}
			Given sequences of real numbers $\left(r_{m}\right)_{m\in\mathbb{N}}$ and $\left(s_{n}\right)_{n\in\mathbb{N}}$ and sequences of measures $\left(\mu_{m}\right)_{m\in\mathbb{N}}$ and $\left(\nu_{n}\right)_{n\in\mathbb{N}}$ in $\mathcal{P}(\mathcal{X})$, let us assume that $r_{m}\,\left(\mu_{m}-\mu\right)$ and $s_{n}\,\left(\nu_{n}-\nu\right)$ converges weakly on $\ell^{\infty}(\mathcal{F})$. Therefore, we aim to study the weak convergence of
			\begin{equation}\label{Eqn:AsympGoal}
				\sqrt{\frac{r_{m}^{2}\,s_{n}^{2}}{r_{m}^{2}+s_{n}^{2}}}\,\left(\Phi\left(\mu_{m},\nu_{n}\right)-\Phi(\mu,\nu)\right).
			\end{equation}
			Our main tool is the (extended) delta method (see \cite{Shapiro1991}, \cite{Dumbgen1993} or \cite{Romisch2004}). Using this approach, the problem of finding a limit for \eqref{Eqn:AsympGoal} is split into two main tasks: Hadamard directional differentiability of the operator $\Phi$, introduced in Definition \ref{Def:GeneralizedColoc}; and the asymptotic study of sequences of estimators $\left(\mu_{m}\right)_{m\in\mathbb{N}}$ and $\left(\nu_{n}\right)_{n\in\mathbb{N}}$.
	\subsection{Hadamard directional differentiability review}
		Given $\mathbb{D},\mathbb{E}$ Banach spaces with norms $\|\cdot\|_{\mathbb{D}}$ and $\|\cdot\|_{\mathbb{E}}$, respectively; $\mathbb{D}_{\Gamma}\subseteq\mathbb{D}$ and $\Gamma:\mathbb{D}_{\Gamma}\longrightarrow\mathbb{E}$, it is said that $\Gamma$ is \emph{Hadamard directional differentiable} at a point $\theta$ tangentially to $\mathbb{D}_{\theta}\subseteq\mathbb{D}$ if and only if there exists a map $\Gamma_{\theta}^{\prime}:\mathbb{D}_{\theta}\longrightarrow\mathbb{E}$ such that
		\begin{equation}
			\left\|\frac{\Gamma\left(\theta+t_{j}\,h_{j}\right)-\Gamma(\theta)}{t_{j}}-\Gamma_{\theta}^{\prime}(h)\right\|_{\mathbb{E}}\longrightarrow0,
		\end{equation}
		when $j\longrightarrow\infty$, where $\left(t_{j}\right)_{j\in\mathbb{N}}\in\mathbb{R}^{\mathbb{N}}$ and $t_{j}\searrow0$; and $\left(h_{j}\right)_{j\in\mathbb{N}}\in\mathbb{D}^{\mathbb{N}}$, $h\in\mathbb{D}_{\theta}$ with $\theta+t_{j}\,h_{j}\in\mathbb{D}_{\Gamma}$ and $h_{j}\longrightarrow h$ in $\mathbb{D}$.
		
		Additionally, it is said that $\Gamma$ is \emph{bounded directional differentiable} at a point $\theta$ tangentially to $\mathbb{D}_{\theta}\subseteq\mathbb{D}$ if and only if there exists a map $\Gamma_{\theta}^{\prime}:\mathbb{D}_{\theta}\longrightarrow\mathbb{E}$ such that
		\begin{equation}
			\sup_{\|h\|_{\mathbb{D}}=1}\left(\left\|\frac{\Gamma(\theta+t\,h)-\Gamma(\theta)}{t}-\Gamma_{\theta}^{\prime}(h)\right\|_{\mathbb{E}}\right)\longrightarrow0,
		\end{equation}
		when $t\searrow0$ in $\mathbb{R}$ and $h\in\mathbb{D}_{\theta}$ with $\theta+t\,h\in\mathbb{D}_{\Gamma}$. When $\mathbb{D}_{\theta}=\mathbb{D}$ and $\Gamma_{\theta}^{\prime}$ is linear, this notion of differentiability is known as \emph{Fréchet differentiability}.
		
		For our purposes, the role of $\Gamma$ will be played by $\Phi$, the domain $\mathbb{D}_{\Gamma}$ will be $\mathcal{P}(\mathcal{X})^{2}$; and $\theta$ will be $(\mu,\nu)$. For the definition of the tangent set $\mathbb{D}_{\theta}$, the domain of the derivative, it has to be fulfilled that the pair $\left(\mu+t\,h^{\mu},\nu+t\,h^{\nu}\right)$ are probability measures for every $t\in[0,\infty)$ and $\left(h^{\mu},h^{\nu}\right)\in\mathbb{D}_{\theta}$. In formal terms, $\mathbb{D}_{\theta}$ is the space of valid perturbations of $(\mu,\nu)$ such that $\left(\mu+t\,h^{\mu},\nu+t\,h^{\nu}\right)\mapsto\pi_{\mu+t\,h^{\mu},\nu+t\,h^{\nu}}^{\lambda}$ is still well defined. A natural choice is $\mathbb{D}_{\theta}=\left(-\mu+\mathcal{P}(\mathcal{X})\right)\times\left(-\nu+\mathcal{P}(\mathcal{X})\right)$. 
		
		Additionally, the topology on $\mathbb{D}_{\theta}$ and $\mathbb{D}_{\Gamma}$ is given by the norm of $\mathbb{D}$, which in our scenario is the Banach space where $\mathcal{P}(\mathcal{X})$ is to be embededded. In this paper, we endow $\mathcal{P}(\mathcal{X})$ with the norm
		\begin{equation}\label{Eqn:MeasureNorm}
			\begin{aligned}
				\left\|\left(\eta_{1},\eta_{2}\right)\right\|_{\mathbb{D}}&=\max\left(\left\|\eta_{1}\right\|_{\ell^{\infty}\left(\mathcal{C}_{0}(\mathcal{X})_{1}\cup\mathcal{F}_{\mu}\right)},\left\|\eta_{2}\right\|_{\ell^{\infty}\left(\mathcal{C}_{0}(\mathcal{X})_{1}\cup\mathcal{F}_{\nu}\right)}\right),
				\\
				\|\eta\|_{\ell^{\infty}\left(\mathcal{C}_{0}(\mathcal{X})_{1}\cup\mathcal{G}\right)}&=\sup_{u\in \mathcal{C}_{0}(\mathcal{X})_{1}\cup\mathcal{G}}\left(\left|\int_{\mathcal{X}}u\,\operatorname{d}\!\eta\right|\right),\quad\eta\in\mathcal{M}(\mathcal{X}),\ \mathcal{G}=\left\{\mathcal{F}_{\mu},\mathcal{F}_{\nu}\right\}.
			\end{aligned}
		\end{equation}
		This norm captures the convergence of measures $\mathcal{M}(\mathcal{X})$ in two spaces. First of them is $\ell^{\infty}\left(\mathcal{C}_{0}(\mathcal{X})_{1}\right)$. Controlling the convergence of measures in terms of this norm guarantees that the bounded directional derivative of the duality mapping $\psi$ exists. More precisely, if \eqref{Eqn:ImplicitRelationPotentials} is seen as the implicit equation $\Psi(\mu,\nu,f,g)=0$ with $(f,g)\in\mathcal{C}_{\operatorname{b}}(\mathcal{X})\times\mathcal{C}_{\operatorname{b}}(\mathcal{X})$, $\int_{\mathcal{X}}g\,\operatorname{d}\!\nu=0$ and
		\begin{equation}\label{Eqn:Psi}
			\Psi(\mu,\nu,f,g)=\left(\begin{array}{l}
				f+\lambda\,\log\left(\int_{\mathcal{X}}\xi(0,g)(x,y)\,\operatorname{d}\!\nu(y)\right)
				\\
				\begin{aligned}
					g&\textstyle+\lambda\,\log\left(\int_{\mathcal{X}}\xi(f,0)(x,y)\,\operatorname{d}\!\mu(x)\right)
					\\
					&\textstyle-\int_{\mathcal{X}}\lambda\,\log\left(\int_{\mathcal{X}}\xi(f,0)(x,y)\,\operatorname{d}\!\mu(x)\right)\operatorname{d}\!\nu(y)
				\end{aligned}
			\end{array}\right),
		\end{equation}
		using the implicit mapping theorem it is proved that $\psi$ exists and it is bounded directional differentiable (on the space $\mathcal{P}(\mathcal{X})^{2}$ endowed with the topology induced by $\ell^{\infty}\left(\mathcal{C}_{0}(\mathcal{X})_{1}\right)^{2}$ ). The second space where the norm refers to are $\ell^{\infty}\left(\mathcal{F}_{\mu}\right)$ and $\ell^{\infty}\left(\mathcal{F}_{\nu}\right)$. As it was seen in Subsection \ref{Subsec:Theory}, those spaces are the ones where the empirical processes $\mathbb{G}_{m}^{\mu}$ and $\mathbb{G}_{n}^{\nu}$ naturally falls within, respectively, in the context of colocalization problem.
		
		Finally, recall that at Definition \ref{Def:GeneralizedColoc}, the transformation $\Phi$ is presented as a mapping from the set of Borel's probability measures $\mathcal{P}(\mathcal{X})$ to the set of functionals on a class $\mathcal{F}\subseteq\operatorname{L}^{1}(\mu\otimes\nu)$. Hence, a natural \emph{universal} choice to play the role of $\mathbb{E}$ is $\ell^{\infty}(\mathcal{F})$ with its natural norm.
		\subsection{Differentiability result}
		In order to provide a general asymptotic result for the (generalized) colocalization mapping, the Hadamard directional differentiability of $\Phi$ is given in the next theorem.
		\begin{Th}\label{Th:DiffPhi}
			Let us assume \ref{Itm:Con}, \ref{Itm:LCm} and \ref{Itm:Bnd}. The generalized colocalization mapping $\Phi$ is Hadamard directional differentiable at any $(\mu,\nu)\in\mathcal{P}(\mathcal{X})^{2}$ tangentially to $\left(-\mu+\mathcal{P}(\mathcal{X})\right)\times\left(-\nu+\mathcal{P}(\mathcal{X})\right)$ with derivative
			\begin{equation}\label{Eqn:DerivativePhi}
				\Phi_{(\mu,\nu)}^{\prime}\left(h^{\mu},h^{\nu}\right)(u)=\mathcal{I}_{1}(u)+\mathcal{I}_{2}(u),
			\end{equation}
			where
			\begin{equation*}
				\begin{aligned}
					&\begin{aligned}
						\mathcal{I}_{1}(u)&=\frac{1}{\lambda}\,\int_{\mathcal{X}^{2}}u(x,y)\,\xi\left(f_{\mu,\nu}^{\lambda},g_{\mu,\nu}^{\lambda}\right)(x,y)\,\left(\psi_{(\mu,\nu)}^{\prime}\left(h^{\mu},h^{\nu}\right)^{(1)}(x)\right.
						\\
						&\quad\left.+\psi_{(\mu,\nu)}^{\prime}\left(h^{\mu},h^{\nu}\right)^{(2)}(y)\right)\,\operatorname{d}\!\mu\otimes\nu(x,y),
					\end{aligned}
					\\
					&\mathcal{I}_{2}(u)=\int_{\mathcal{X}^{2}}u(x,y)\,\xi\left(f_{\mu,\nu}^{\lambda},g_{\mu,\nu}^{\lambda}\right)(x,y)\,\operatorname{d}\!\left(\mu\otimes h^{\nu}+h^{\mu}\otimes\nu\right)(x,y),
				\end{aligned}
			\end{equation*}
			where the superindex $^{(i)}$ for $i\in\{1,2\}$ denotes the projection on the $i$-th coordinate.
		\end{Th}
		It is worth to make some comments before proving this result. Firstly, as it is highlighted in expression \eqref{Eqn:DerivativePhi}, the derivative of $\Phi$ is constituted by two parts. In the first term $\mathcal{I}_{1}$, the duality mapping $\psi$ is evaluated on the pair of measures $\mu$ and $\nu$. Hence, as $\psi$ is an implicit function, this part of the proof strongly depends on the differentiability properties of the map $\Psi$ in \eqref{Eqn:Psi}. In the second one $\mathcal{I}_{2}$, the measures $\mu$ an $\nu$ act as integral functionals on classes $\mathcal{F}_{\mu}$ and $\mathcal{F}_{\nu}$, respectively. Since $\Phi$ is an integral operator where the integrand and the measure depend on $(\mu,\nu)$ it is not surprising that the proof of Theorem \ref{Th:DiffPhi} mainly consists on interchange the derivative (in the space of measures) and the integral.
		
		Usually, interchanging derivative and integral results is strongly related to the dominated convergence theorem. In the following lemmas is seen that\break$\xi\left(f_{\mu,\nu}^{\lambda},g_{\mu,\nu}^{\lambda}\right)(x,y)$ is bounded differentiable and the derivative satisfies the necessary requirements for this result to be applied. To begin with, we will obtain a known bound on the potentials $f_{\mu,\nu}^{\lambda}$ and $g_{\mu,\nu}^{\lambda}$ under the convention that $\int_{\mathcal{X}}g_{\mu,\nu}^{\lambda}\,\operatorname{d}\!\nu=0$.
		\begin{Lema}\label{Lema:BoundsPotentials}
			Assume \ref{Itm:Con}. For all $x,y\in\mathcal{X}$, we have that
			\begin{equation}
				\begin{aligned}
					&\inf_{y\in\mathcal{X}}\left(c(x,y)-g_{\mu,\nu}^{\lambda}(y)\right)\leq f_{\mu,\nu}^{\lambda}(x)\leq\int_{\mathcal{X}}c(x,y)\,\operatorname{d}\!\nu(y),
					\\
					&\inf_{x\in\mathcal{X}}\left(c(x,y)-f_{\mu,\nu}^{\lambda}(x)\right)\leq g_{\mu,\nu}^{\lambda}(y)\leq\int_{\mathcal{X}}c(x,y)\,\operatorname{d}\!\mu(x).
				\end{aligned}
			\end{equation}
			In particular, $\left\|f_{\mu,\nu}^{\lambda}\right\|_{\infty},\left\|g_{\mu,\nu}^{\lambda}\right\|_{\infty}\leq\|c\|_{\infty}$.
		\end{Lema}
		\begin{proof}
			Our proof follows from the same lines of \cite[Lemma 4.9]{Nutz2021}. Let us start with the upper bounds. By Jensen's inequality
			\begin{equation}
				\begin{aligned}
					f_{\mu,\nu}^{\lambda}(x)&=-\lambda\,\log\left(\int_{\mathcal{X}}\exp\left(\frac{g_{\mu,\nu}^{\lambda}(y)-c(x,y)}{\lambda}\right)\,\operatorname{d}\!\nu(y)\right)
					\\
					&\leq\int_{\mathcal{X}}\left(c(x,y)-g_{\mu,\nu}^{\lambda}(y)\right)\,\operatorname{d}\!\nu(y).
				\end{aligned}
			\end{equation}
			Now recall that $\int_{\mathcal{X}}g_{\mu,\nu}^{\lambda}\,\operatorname{d}\!\nu(y)=0$, so the first upper bound is obtained. Analogously,
			\begin{equation}
				g_{\mu,\nu}^{\lambda}(y)\leq\int_{\mathcal{X}}\left(c(x,y)-f_{\mu,\nu}^{\lambda}(x)\right)\,\operatorname{d}\!\mu(x).
			\end{equation}
			Looking at the dual formulation of the entropic optimal transport problem \eqref{Eqn:DualEOT}, we have that $\int_{\mathcal{X}}f_{\mu,\nu}^{\lambda}(x)\,\operatorname{d}\!\mu(x)=\operatorname{EOT}_{c}^{\lambda}(\mu,\nu)>0$, so the right hand side of the last inequality can be bounded by $\int_{\mathcal{X}}c(x,y)\,\operatorname{d}\!\mu(x)$ and the upper bound for $g_{\mu,\nu}^{\lambda}$ is obtained, too. Further, from these upper bounds, the last line of the lemma is immediate by definition of supremum norm.
			
			For the lower bounds, take infima on the exponents of \eqref{Eqn:ImplicitRelationPotentials}.
		\end{proof}
		It is important to mention that in the previous proof, the continuity hypothesis on \ref{Itm:Con} is not used.
		\begin{Lema}\label{Lema:ourDiffDens}
			Assume \ref{Itm:Con}.
			\begin{enumerate}
				\item Let $f,g\in\mathcal{C}_{\operatorname{b}}(\mathcal{X})$ with $\|f\|_{\infty},\|g\|_{\infty}\leq\|c\|_{\infty}$ and $\int_{\mathcal{X}}g\,\operatorname{d}\!\nu=0$. Denote, with abuse of notation, by $\xi$ the operator on $\mathcal{C}_{\operatorname{b}}(\mathcal{X})\times\mathcal{C}_{\operatorname{b}}(\mathcal{X})$ defined by $\xi(f,g)(x,y)$, where $x,y\in\mathcal{X}$. Then $\xi$ is Fréchet differentiable with derivative $D\xi_{(f,g)}\left(h^{f},h^{g}\right)(x,y)=\xi(f,g)(x,y)\,\frac{h^{f}(x)+h^{g}(y)}{\lambda}$.
				\item Further, assume \ref{Itm:LCm}. Let $\mu,\nu\in\mathcal{P}(\mathcal{X})$. Then, the duality mapping $\psi$ is bounded directional differentiable tangentially to $\left(-\mu+\mathcal{P}(\mathcal{X})\right)$\break$\times\left(-\nu+\mathcal{P}(\mathcal{X})\right)$ with derivative  $\psi_{(\mu,\nu)}^{\prime}=-\partial_{2}\Psi_{(\mu,\nu,\psi(\mu,\nu))}^{-1}\,\partial_{1}\Psi_{(\mu,\nu,\psi(\mu,\nu))}$.
				
				Define $\Xi(\mu,\nu)=\xi\left(f_{\mu,\nu}^{\lambda},g_{\mu,\nu}^{\lambda}\right)$. $\Xi$ is bounded directional differentiable tangentially to $\left(-\mu+\mathcal{P}(\mathcal{X})\right)\times\left(-\nu+\mathcal{P}(\mathcal{X})\right)$ with derivative
				\begin{equation}
					\begin{aligned}
						&\Xi_{(\mu,\nu)}\left(h^{\mu},h^{\nu}\right)(x,y)=\xi\left(f_{\mu,\nu}^{\lambda},g_{\mu,\nu}^{\lambda}\right)(x,y)
						\\
						&\quad\cdot\frac{\psi_{(\mu,\nu)}^{\prime}\left(h^{\mu},h^{\nu}\right)^{(1)}(x)+\psi_{(\mu,\nu)}^{\prime}\left(h^{\mu},h^{\nu}\right)^{(2)}(y)}{\lambda},
					\end{aligned}
				\end{equation}
				where the superindex $^{(i)}$ for $i\in\{1,2\}$ denotes the projection on the $i$-th coordinate.
			\end{enumerate}
		\end{Lema}
		The previous lemma should be read as a functional differentiability result. Nevertheless, the differentiability of the map $\psi$ needs further contextualization. Firstly, the map $\psi$, as it is done in this paper too, is usually presented as map from $\mathcal{P}(\mathcal{X})^{2}$ to the space $\mathcal{C}_{\operatorname{b}}(\mathcal{X})\times\mathcal{C}_{\operatorname{b}}(\mathcal{X})$ (provided \ref{Itm:Con} holds). A similar result is presented in \cite{Carlier&Laborde2020} where the differentiability of $\psi$ is proved under different assumptions. In particular, for every pair $(\mu,\nu)\in\mathcal{P}(\mathcal{X})^{2}$ that are absolutely continuos respect to a pre-frixed pair $\left(m_{1},m_{2}\right)\in\mathcal{P}(\mathcal{X})^{2}$, the same conclusion is obtained. Lemma \ref{Lema:ourDiffDens} is a generalization of that differentiability result provided \ref{Itm:LCm} is given. Another interesting work in this line is presented in \cite{Goldfeld2024}, where the same formula for the derivative of $\Xi$ is obtained using an alternative norm. Nevertheless, the notion of convergence considered there does not meet the requirements needed for our purposes in this work.
		\begin{proof}
			\begin{enumerate}
				\item By definition of Fréchet differentiability, the goal is to show that given $f,g\in\mathcal{C}_{\operatorname{b}}(\mathcal{X})$ for every $h^{f},h^{g}\in\mathcal{C}_{\operatorname{b}}(\mathcal{X})$ with $\max\left(\left\|h^{f}\right\|_{\infty},\left\|h^{g}\right\|_{\infty}\right)=1$ that
				\begin{equation}\label{Eqn:FrechetDiffDef}
					\begin{aligned}
						&\sup_{x,y\in\mathcal{X}}\left(\left|\frac{\xi\left(f+t\,h^{f},g+t\,h^{g}\right)(x,y)-\xi(f,g)(x,y)}{t}\right.\right.
						\\
						&\quad\left.\left.-\xi(f,g)\,\frac{h^{f}(x)+h^{g}(y)}{\lambda}\right|\right)\longrightarrow0,
					\end{aligned}
				\end{equation}
				given $t\searrow0$. Now by Lemma \ref{Lema:BoundsPotentials}, \eqref{Eqn:FrechetDiffDef} can be bounded by
				\begin{equation}\label{Eqn:Lema5Aux1}
					\textstyle\exp\left(\frac{2\,\|c\|_{\infty}}{\lambda}\right)\,\underset{x,y\in\mathcal{X}}{\operatorname{sup}}\left(\left|\frac{1}{t}\,\left(\exp\left(\frac{t\,\left(h^{f}(x)+h^{g}(y)\right)}{\lambda}\right)-1-t\,\frac{h^{f}(x)+h^{g}(y)}{\lambda}\right)\right|\right)
				\end{equation}
				when taking out common factor $\xi(f,g)$ and applying Lemma \ref{Lema:BoundsPotentials}. Since the argument of the supremum of \eqref{Eqn:Lema5Aux1} is differentiable on $t$, mean value theorem provides that the previous expression is equal to
				\begin{equation}
					\textstyle\exp\left(\frac{2\,\|c\|_{\infty}}{\lambda}\right)\,\underset{x,y\in\mathcal{X}}{\operatorname{sup}}\left(\left|\exp\left(\frac{t^{\ast}(x,y)\,\left(h^{f}(x)+h^{g}(y)\right)}{\lambda}\right)-1\right|\,\frac{\left|h^{f}(x)+h^{g}(y)\right|}{\lambda}\right)
				\end{equation}
				where $t^{\ast}(x,y)\in[0,t]$. By definition of $t^{\ast}$,$h^{f}$ and $h^{g}$,
				\begin{equation}
					1-\exp\left(-\frac{2\,t}{\lambda}\right)\leq\exp\left(\frac{t^{\ast}(x,y)\,\left(h^{f}(x)+h^{g}(y)\right)}{\lambda}\right)-1\leq\exp\left(\frac{2\,t}{\lambda}\right)-1,
				\end{equation}
				and the derivative is obtained.
				\item To prove this result the implicit mapping theorem in Banach spaces is to be used (see \cite[Section 4.8]{Zeidler2012}). As it was pointed in Section \ref{Sec:Intro}, the map $\psi$ is well defined by existence and uniqueness of solution theorem for the entropic optimal transport problem \eqref{Eqn:EOT} (see \cite{Nutz2021}). More precisely, for fixed $\mu,\nu\in\mathcal{P}(\mathcal{X})$ there exists a pair $\left(f_{\mu,\nu}^{\lambda},g_{\mu,\nu}^{\lambda}\right)\in\operatorname{L}^{1}(\mu)\times\operatorname{L}^{1}(\nu)$ with $\int_{\mathcal{X}}g_{\mu,\nu}^{\lambda}\,\operatorname{d}\!\nu=0$. Further, by \ref{Itm:Con} they are also bounded and continuous. In terms of the implicit mapping $\Psi$: $\Psi\left(\mu,\nu,f_{\mu,\nu}^{\lambda},g_{\mu,\nu}^{\lambda}\right)=0$.
				
				Now, it is worth to observe that by \ref{Itm:LCm}, every Borel probability measure is, in fact, a Radon measure (see \cite[Theorem 7.8]{Folland2013}). Hence, without loss of generality we can extend $\Psi$ from $\mathcal{P}(\mathcal{X})^{2}$ to $\left(\mathcal{C}_{0}(\mathcal{X})^{\ast}\right)^{2}$, that is, the (linear) dual space of $\mathcal{C}_{0}(\mathcal{X})$. Note that the embedding of $\mathcal{P}(\mathcal{X})$ is isometric (with the norm $\|\cdot\|_{\ell^{\infty}\left(\mathcal{C}_{0}(\mathcal{X})_{1}\right)}$). With abuse of notation we will denote this extension by $\Psi$, too.
				
				The next condition of implicit mapping theorem to be checked is the continuous differentiability of $\Psi$. Firstly, the Gâteaux derivative of $\Psi$ takes the form
				\begin{equation}\label{Eqn:DiffPsi}
					\begin{aligned}
						&D\Psi_{(\eta,\zeta,f,g)}\left(h^{\eta},h^{\zeta},h^{f},h^{g}\right)
						\\
						&\quad=\lambda\,\left(\begin{array}{l}
							\frac{\int_{\mathcal{X}}\xi(0,g)(\cdot,y)\,\operatorname{d}\!h^{\zeta}(y)}{\int_{\mathcal{X}}\xi(0,g)(\cdot,y)\,\operatorname{d}\!\zeta(y)}
							\\
							\begin{aligned}
								&\textstyle\frac{\int_{\mathcal{X}}\xi(f,0)(x,\cdot)\,\operatorname{d}\!h^{\eta}(x)}{\int_{\mathcal{X}}\xi(f,0)(x,\cdot)\,\operatorname{d}\!\eta(x)}
								\\
								&\textstyle\quad-\int_{\mathcal{X}}\log\left(\int_{\mathcal{X}}\xi(f,0)(x,y)\,\operatorname{d}\!\eta(x)\right)\,\operatorname{d}\!h^{\zeta}(y)
								\\
								&\textstyle\quad-\int_{\mathcal{X}}\frac{\int_{\mathcal{X}}\xi(f,0)(x,y)\,\operatorname{d}\!h^{\eta}(x)}{\int_{\mathcal{X}}\xi(f,0)(x,y)\,\operatorname{d}\!\eta(x)}\,\operatorname{d}\!\zeta(y)
							\end{aligned}
						\end{array}\right)
						\\
						&\qquad+\left(\begin{array}{l}
							h^{f}+\frac{\int_{\mathcal{X}}\xi(0,g)(\cdot,y)\,h^{g}(y)\operatorname{d}\!\zeta(y)}{\int_{\mathcal{X}}\xi(0,g)(\cdot,y)\,\operatorname{d}\!\zeta(y)}
							\\
							\begin{aligned}
								&\textstyle h^{g}+\frac{\int_{\mathcal{X}}\xi(f,0)(x,\cdot)\,h^{f}(x)\,\operatorname{d}\!\eta(x)}{\int_{\mathcal{X}}\xi(f,0)(x,\cdot)\,\operatorname{d}\!\eta(x)}
								\\
								&\textstyle\quad-\int_{\mathcal{X}}\frac{\int_{\mathcal{X}}\xi(f,0)(x,y)\,h^{f}(x)\,\operatorname{d}\!\eta(x)}{\int_{\mathcal{X}}\xi(f,0)(x,y)\,\operatorname{d}\!\eta(x)}\,\operatorname{d}\!\zeta(y)
							\end{aligned}
						\end{array}\right).
					\end{aligned}
				\end{equation}
				Hence, the continuity in a neighborhood of $(\mu,\nu)$ is just an example of continuity under the integral together with Lemma \ref{Lema:BoundsPotentials}. Gâteaux differentiability in combination with the continuity of the derivative implies Fréchet differentiability.
				
				Additionally, observe that in \eqref{Eqn:DiffPsi} we have split the expression of $D\Psi_{(\eta,\zeta,f,g)}$ in two terms: the partial derivatives $\partial_{1}\Psi_{(\eta,\zeta,f,g)}$ and $\partial_{2}\Psi_{(\eta,\zeta,f,g)}$. The use of the implicit function theorem for $\Psi$ requires the invertibility of
				\begin{equation}\label{Eqn:ParfialDiffPsi2}
					\begin{aligned}
						&\partial_{2}\Psi_{\left(\mu,\nu,f_{\mu,\nu}^{\lambda},g_{\mu,\nu}^{\lambda}\right)}\left(h^{f},h^{g}\right)=\left(\begin{array}{c}
							h^{f}
							\\
							h^{g}
						\end{array}\right)
						\\
						&\quad + \left(\begin{array}{l}
							\int_{\mathcal{X}}\xi\left(f_{\mu,\nu}^{\lambda},g_{\mu,\nu}^{\lambda}\right)(\cdot,y)\,h^{g}(y)\,\operatorname{d}\!\nu(y)
							\\
							\int_{\mathcal{X}}\left(\xi\left(f_{\mu,\nu}^{\lambda}g_{\mu,\nu}^{\lambda}\right)(x,\cdot)-1\right)\,h^{f}(x)\,\operatorname{d}\!\mu(x)
						\end{array}\right).
					\end{aligned}
				\end{equation}
				The operator in \eqref{Eqn:ParfialDiffPsi2} has been proved to be invertible in \cite[Lemma 2.7]{Gonzalez2022} using the Fredholm's Alternative Theorem. Hence, we can apply the implicit mapping theorem in Banach spaces to conclude the differentiability of $\psi$. In particular, if we restrict the domain of $\psi_{(\mu,\nu)}^{\prime}$ to $\left(-\mu+\mathcal{P}(\mathcal{X})\right)\times\left(-\nu+\mathcal{P}(\mathcal{X})\right)$ the desired derivative is obtained.
				
				Finally, differentiability of $\Xi(\mu,\nu)=\xi\left(f_{\mu,\nu}^{\lambda},g_{\mu,\nu}^{\lambda}\right)$ constitutes a straightforward application of the chain rule. Details are omitted.
			\end{enumerate}
		\end{proof}
		Now, we are ready to give the proof of Theorem \ref{Th:DiffPhi}.
		\begin{proof}[Proof of Theorem \ref{Th:DiffPhi}]
			By definition of Hadamard directional differentiability, the goal of this proof is: given measures $h^{\mu}\in-\mu+\mathcal{P}(\mathcal{X})$ and $h^{\nu}\in-\nu+\mathcal{P}(\mathcal{X})$; sequences of measures $\left(h_{j}^{\mu}\right)_{j\in\mathbb{N}}\in\left(-\mu+\mathcal{P}(\mathcal{X})\right)^{\mathbb{N}}$ and $\left(h_{j}^{\nu}\right)_{n\in\mathbb{N}}\in\left(-\nu+\mathcal{P}(\mathcal{X})\right)^{\mathbb{N}}$ such that $\left\|h_{j}^{\mu}-h^{\mu}\right\|_{\ell^{\infty}\left(\mathcal{C}_{0}(\mathcal{X})_{1}\cup\mathcal{F}_{\mu}\right)},\left\|h_{j}^{\nu}-h^{\nu}\right\|_{\ell^{\infty}\left(\mathcal{C}_{0}(\mathcal{X})_{1}\cup\mathcal{F}_{\nu}\right)}\longrightarrow0$ and a sequence of real numbers $t_{j}$ decreasing to $0$ when $j\longrightarrow\infty$; showing
			\begin{equation}
				\lim_{j\longrightarrow\infty}\frac{\Phi\left(\mu+t_{j}\,h_{j}^{\mu},\nu+t_{j}\,h_{j}^{\nu}\right)-\Phi(\mu,\nu)-t_{j}\,\Phi_{(\mu,\nu)}^{\prime}\left(h^{\mu},h^{\nu}\right)}{t_{j}}=0,
			\end{equation}
			in $\ell^{\infty}\left(\mathcal{F}\right)$.
			
			To begin with, observe that
			\begin{equation}
				\begin{aligned}
					&\frac{\Phi\left(\mu+t_{j}\,h_{j}^{\mu},\nu+t_{j}\,h_{j}^{\nu}\right)(u)-\Phi(\mu,\nu)(u)-t_{j}\,\Phi_{(\mu,\nu)}^{\prime}\left(h^{\mu},h^{\nu}\right)(u)}{t_{j}}
					\\
					&\quad=\mathcal{T}_{1}+\mathcal{T}_{2}+\mathcal{T}_{3},
				\end{aligned}
			\end{equation}
			where
			\begin{equation}\label{Eqn:Monster}
				\begin{aligned}
					\mathcal{T}_{1}&=\int_{\mathcal{X}^{2}}u(x,y)\,\left(\frac{\xi\left(f_{j}^{\lambda},g_{j}^{\lambda}\right)(x,y)-\xi\left(f_{\mu,\nu}^{\lambda},g_{\mu,\nu}^{\lambda}\right)(x,y)}{t_{j}}\right.
					\\
					&\quad-\frac{1}{\lambda}\,\xi\left(f_{\mu,\nu}^{\lambda},g_{\mu,\nu}^{\lambda}\right)(x,y)\,\left(\psi_{(\mu,\nu)}^{\prime}\left(h^{\mu},h^{\nu}\right)^{(1)}(x)\right.
					\\
					&\quad\left.\left.+\psi_{(\mu,\nu)}^{\prime}\left(h^{\mu},h^{\nu}\right)^{(2)}(y)\right)\right)\,\operatorname{d}\mu\otimes\nu(x,y),
					\\
					\mathcal{T}_{2}&=\int_{\mathcal{X}^{2}}u(x,y)\,\xi\left(f_{j}^{\lambda},g_{j}^{\lambda}\right)(x,y)\,\operatorname{d}\left(\mu\otimes h_{j}^{\nu}+h_{j}^{\mu}\otimes\nu\right)(x,y),
					\\
					&\quad-\int_{\mathcal{X}^{2}}u(x,y)\,\xi\left(f_{\mu,\nu}^{\lambda},g_{\mu,\nu}^{\lambda}\right)(x,y)\,\operatorname{d}\left(\mu\otimes h^{\nu}+h^{\mu}\otimes\nu\right)(x,y)
					\\
					\mathcal{T}_{3}&=t_{j}\int_{\mathcal{X}^{2}}u(x,y)\,\xi\left(f_{j}^{\lambda},g_{j}^{\lambda}\right)(x,y)\,\operatorname{d}h_{j}^{\mu}\otimes h_{j}^{\nu}(x,y),
				\end{aligned}
			\end{equation}
			where $\left(f_{j}^{\lambda},g_{j}^{\lambda}\right)=\psi\left(\mu+t_{j}\,h_{j}^{\mu},\nu+t_{j}\,h^{\nu}\right)$. Now, we will split the proof in three parts.
			
			Firstly, remind that the Radon-Nykodim's derivative of $\pi_{\mu,\nu}^{\lambda}$ respect to $\mu\otimes\nu$, $\xi\left(f_{\mu,\nu}^{\lambda},g_{\mu,\nu}^{\lambda}\right)$, is bounded directional differentiable in the uniform norm (see Lemma \ref{Lema:ourDiffDens}). By dominated convergence theorem, $\mathcal{T}_{1}$ in \eqref{Eqn:Monster} is going to $0$. Note that this convergence is uniformly in $u\in\mathcal{F}$ by \ref{Itm:Bnd}.
			
			Secondly, $\mathcal{T}_{2}$ is equal to the sum of integrals
			\begin{equation}
				\begin{aligned}
					&\int_{\mathcal{X}^{2}}u(x,y)\,\left(\xi\left(f_{j}^{\lambda},g_{j}^{\lambda}\right)(x,y)-\xi\left(f_{\mu,\nu}^{\lambda},g_{\mu,\nu}^{\lambda}\right)(x,y)\right)\,\operatorname{d}\left(\mu\otimes h_{j}^{\nu}+h_{j}^{\mu}\otimes\nu\right)(x,y)
					\\
					&\quad+\int_{\mathcal{X}^{2}}u(x,y)\,\xi\left(f_{\mu,\nu}^{\lambda},g_{\mu,\nu}^{\lambda}\right)(x,y)\,\operatorname{d}\left(\mu\otimes\left(h_{j}^{\nu}-h^{\nu}\right)+\left(h_{j}^{\mu}-h^{\nu}\right)\otimes\nu\right)(x,y).
				\end{aligned}
			\end{equation}
			First term is converging to $0$ by the continuity of the duality mapping in the uniform norm (uniformly in $u\in\mathcal{F}$ by \ref{Itm:Bnd}). Second term is, precisely, the sum of $\left\|h_{j}^{\mu}-h^{\mu}\right\|_{\ell^{\infty}\left(\mathcal{F}_{\mu}\right)}$ and $\left\|h_{j}^{\nu}-h^{\nu}\right\|_{\ell^{\infty}\left(\mathcal{F}_{\nu}\right)}$, that goes to $0$ by hypothesis.
			
			Finally, since $t_{j}\searrow0$, bounding the integral in $\mathcal{T}_{3}$, the remainder, uniformly in $j$ is enough to see that the last term is converging to zero. By Lemma \ref{Lema:BoundsPotentials}, the integrand is bounded by $\exp\left(\frac{2\,\|c\|_{\infty}}{\lambda}\right)$. Hence, in order to bound the integral just observe that, by definition, $h_{j}^{\mu}$ and $h_{j}^{\nu}$ are differences of probability measures and its total variation is less or equal than $2$. Putting all together, $\mathcal{T}_{3}$ is tending to zero and this proof is ended.
		\end{proof}
	\subsection{Donsker property}
		In this subsection we focus on the proof of Theorem \ref{Th:Donskerity}. Even thought the nature of the sequences of estimators $\left(\mu_{m}\right)_{m\in\mathbb{N}}$ and $\left(\nu_{n}\right)_{n\in\mathbb{N}}$ is irrelevant to establish the asymptotic result, in the majority of applications the estimator is the empirical measure. In particular, on this work we have focus on the limit behavior of plug-in estimator $\Phi\left(\mu_{m},\nu_{n}\right)$ when the $\mu_{m}$ and $\nu_{n}$ are the empirical measures associated to given samples. So, it is worth to provide sufficient conditions for \ref{Itm:Dnk}. We start with a preliminary lemma.
		\begin{Lema}\label{Lema:AuxiliarIndicators}
			Let $\eta\in\mathcal{P}(\mathcal{X})$ and $\mathcal{G}$ be a class of measurable functions defined as follows: given $M:\mathcal{X}\times[0,1]\longrightarrow\mathbb{R}$ monotone for each $x$,
			$\mathcal{G}=\left\{M(\cdot,r):r\in[0,1]\right\}$.
			If the envelope function of $\mathcal{G}$ is in $\operatorname{L}^{2}(\eta)$, then the bracketing numbers $N_{[\,]}\left(\varepsilon,\mathcal{G},\operatorname{L}^{2}(\eta)\right)$ are polynomial.
		\end{Lema}
		\begin{proof}
			For this proof, assume without loss of generality that $M$ is right continuous and that $M(\cdot,0)\equiv0$. By monotonicity, it is clear that $M(x,\cdot)$ is a càdlàg function and that the envelope function of the class $\mathcal{G}$ is $M(\cdot,1)$. Observe that brackets with extremes in $\mathcal{G}$ take the form
			\begin{equation}
				[M(\cdot,a),M(\cdot,b)]=\{M(\cdot,s):s\in[a,b]\},
			\end{equation}
			and $\varepsilon$-brackets satisfy that $\left(\int_{\mathcal{X}}(M(x,b)-M(x,a))^{2}\,\operatorname{d}\!\eta(x)\right)^{^{1}\!/\!_{2}}<\varepsilon$.
			
			Define
			\begin{equation}
				G(r)=\left(\int_{\mathcal{X}}M(x,r)^{2}\,\operatorname{d}\!\eta(x)\right)^{^{1}\!/\!_{2}}.
			\end{equation}
			It follows that $G$ is monotone, c\`{a}dl\`{a}g by Lebesgue's dominated convergence theorem and bounded by $G(1)$ by hypothesis. Take $\varepsilon>0$ such that $\varepsilon<4\,\sqrt{2}\,G(1)$. Call $r_{l}=G^{-1}\left(l\,\frac{\varepsilon^{2}}{32\,G(1)}\right)$ for $l\in\{1,\ldots,L\}$ with $L=\left\lfloor\frac{32\,G(1)^{2}}{\varepsilon^{2}}\right\rfloor$, where the inverse is understood in the quantile way. We claim that the colections of brackets\break$\left\{\left[M\left(\cdot,r_{l-1}\right),M\left(\cdot,r_{l}^{-}\right)\right]\right\}_{l=1,\ldots,L}$ covers $\mathcal{G}$ and the length of every bracket is smaller than $\varepsilon$; where it is understood that $r_{0}=0$ and
			\begin{equation}
				M\left(x,r^{-}\right)=\lim_{s\longrightarrow r^{-}}M(x,s),\quad x\in\mathcal{X}.
			\end{equation}
			The first part of the statement is straightforward by construction of the brackets. Let us proof the second one:
			\begin{equation}
				\begin{aligned}
					\left\|M\left(\cdot,r_{l}^{-}\right)-M\left(\cdot,r_{l-1}\right)\right\|_{\operatorname{L}^{2}(\eta)}&\leq\left(G\left(r_{l}\right)^{2}-G\left(r_{l-1}\right)^{2}\right)^{\rfrac{1}{2}}
					\\
					&\leq2\,\sqrt{G(1)}\,\left(G\left(r_{l}\right)-G\left(r_{l-1}\right)\right)^{\rfrac{1}{2}}\leq\frac{\varepsilon}{2}<\varepsilon.
				\end{aligned}
			\end{equation}
			
			Since the number of brackets needed is $\left\lfloor\frac{32\,G(1)^{2}}{\varepsilon^{2}}\right\rfloor$, we conclude that $N_{[\,]}\left(\varepsilon,\mathcal{G},\operatorname{L}^{2}(\eta)\right)$ is polynomial.
		\end{proof}
	\subsection{Proof of main results}
		\begin{proof}[Proof of Theorem \ref{Th:Donskerity}]
			The proof of \ref{Itm:Dnk} for the class $\mathcal{F}_{I,c}$ is a direct application of Lemma \ref{Lema:AuxiliarIndicators}. The proof for $\mathcal{C}_{\operatorname{b}}^{s}\left(\mathcal{X}^{2}\right)_{1}$ lies in the following fact: when $\mathcal{F}=\mathcal{C}_{\operatorname{b}}^{s}\left(\mathcal{X}^{2}\right)_{1}$, $\mathcal{F}_{\mu}$ and $\mathcal{F}_{\nu}$ are bounded subsets of $\mathcal{C}_{\operatorname{b}}^{s}(\mathcal{X})$. Then, \cite[Theorem 2.7.1]{van_der_Vaart&Wellner2023} concludes the proof.
		\end{proof}
		As it is shown in \cite{Shapiro1991}, Hadamard directional differentiability is the appropriate one for the delta method. It is also known that bounded directional differentiability implies Hadamard directional differentiability and the latter implies Gâteaux differentiability. Additionally, in finite dimensional topological vector spaces, Hadamard and bounded differentiability are equivalent. Further details about the distinct notions of differentiability used in statistics and optimization are reviewed in \cite{Shapiro1990} and references therein.
		\begin{proof}[Proof of Theorem \ref{Th:ColocConv}]
			The proof of this result is now a straightforward application of the (extended) delta method presented in \cite[Theorem 2.1]{Shapiro1991} (see also \cite[Theorem 3.10.4]{van_der_Vaart&Wellner2023}). By \ref{Itm:Dnk}, for some $s>0$, classes $\mathcal{C}_{0}^{s}(\mathfrak{X})_{1}\cup\mathcal{F}_{\mu}$ and $\mathcal{C}_{0}^{s}(\mathfrak{X})_{1}\cup\mathcal{F}_{\nu}$ are Donsker. Then, by \ref{Itm:Bal}, the product process $\sqrt{\frac{m\,n}{m+n}}\,\left(\mu_{m}-\mu,\nu_{n}-\nu\right)$ converges weakly to $\left(\sqrt{1-\tau}\,\mathbb{G}^{\mu},\sqrt{\tau}\,\mathbb{G}^{\nu}\right)$. By Theorem \ref{Th:DiffPhi}, we have the Hadamard directional differentiability of $\Phi$. By the delta method, the proof is ended.
		\end{proof}
		To finish this subsection, let us mention that Corollary \ref{Cor:Th:ClassicColocConv} is the particular case of Theorem \ref{Th:ColocConv} where $\mathcal{F}=\mathcal{F}_{I,c}$. Details are omitted.
	\subsection{Consistency of empirical results} \label{SuplMat:Consistency}
		In Section \ref{SubSec:RealDataAnalysis}, the colocalization curve is used to analyze a real dataset. However, in that context, the data are not given as a sample over the space $\mathcal{X}$, but rather as a smoothed version defined over a partition of it. In this subsection, we provide consistency results for this possible scenario, which is quite common in the literature (see \textcite{Schrieber2016}).
		
		To begin with, let us introduce some notation for this subsection. Given $\eta\in\mathcal{P}(\mathcal{X})$, assume that there exists a family of Borel sets $\left\{\mathcal{R}_{i}^{\eta}\right\}_{i=1,\ldots,L}$ such that $\mathcal{X}=\bigcup_{i=1}^{L}\mathcal{R}_{i}^{\eta}$, $\mathcal{R}_{i}^{\eta}\cap\mathcal{R}_{j}^{\eta}=\emptyset$, and $\eta\left(\mathcal{R}_{i}^{\eta}\right) \geq 0$. Further, there exists $\left\{U_{i,l}^{\eta}\right\}_{i=1,\ldots,L}$ that satisfies $U_{i,l}^{\eta} \geq 0$ and $\sum_{i=1}^{L}U_{i,l}^{\eta}=1$. Denote  $\mathbf{m}^{\eta}=\left(\eta\left(\mathcal{R}_{i}^{\eta}\right)\right)_{i=1,\ldots,L}$ and $\mathbb{U}_{l}^{\eta}=\left(U_{i,l}^{\eta}\right)_{i=1,\ldots,L}$ (all the vectors are understood as column matrices).
		
		In section \ref{SubSec:RealDataAnalysis}, for the measures $\mu$ and $\nu$, families of sets $\left\{\mathcal{R}_{i}^{\mu}\right\}_{i=1,\ldots,L}$ and\break$\left\{\mathcal{R}_{i}^{\nu}\right\}_{i=1,\ldots,L}$ are the grids of pixels while the families of variables $\left\{U_{i,m}^{\mu}\right\}_{i=1,\ldots,L}$ and $\left\{U_{i,n}^{\nu}\right\}_{i=1,\ldots,L}$ model the intensity obtained at the experiment in each of the regions. Hence, the assumption on the asymptotic behaviour imposed on $\left\{U_{i,l}^{\eta}\right\}_{i=1,\ldots,L}$ is reasonable. A particular scenario of this framework is given when samples $X_{1},\ldots,X_{m}\overset{\operatorname{i.i.d}}{\sim}\mu$ and $Y_{1},\ldots,Y_{n}\overset{\operatorname{i.i.d}}{\sim}\nu$ are taken, defining
		\begin{equation}\label{Eqn:NaiveU}
			\begin{aligned}
				U_{i,m}^{\mu}&=\frac{\#\left\{j\in\{1,\ldots,m\}:X_{j}\in\mathcal{R}_{i}^{\mu}\right\}}{m}=\frac{1}{m}\,\sum_{j=1}^{m}\mathbf{1}_{\mathcal{R}_{i}^{\mu}}\left(X_{j}\right)=\mu_{m}\left(\mathbf{1}_{\mathcal{R}_{i}^{\mu}}\right),
				\\
				U_{i,n}^{\nu}&=\frac{\#\left\{j\in\{1,\ldots,n\}:Y_{j}\in\mathcal{R}_{i}^{\nu}\right\}}{n}=\frac{1}{n}\,\sum_{j=1}^{n}\mathbf{1}_{\mathcal{R}_{i}^{\nu}}\left(Y_{j}\right)=\nu_{n}\left(\mathbf{1}_{\mathcal{R}_{i}^{\nu}}\right).
			\end{aligned}
		\end{equation}
		
		Call $\mathcal{F}_{\eta}^{\operatorname{pixel}}=\left\{\mathbf{1}_{\mathcal{R}_{i}^{\eta}}\right\}_{i=1,\ldots,L}$ and define for the measure $\eta$
		\begin{equation}
			\widehat{\eta}_{L}=\sum_{i=1}^{L}\eta\left(\mathcal{R}_{i}^{\eta}\right)\,\delta_{h_{i}^{\eta}},\quad\widehat{\eta}_{L,l}=\sum_{i=1}^{L}U_{i,l}^{\eta}\,\delta_{h_{i}^{\eta}},
		\end{equation}
		where $\left\{h_{i}^{\eta}\right\}_{i=1,\ldots,L}$ are points such that $h_{i}^{\eta}\in\mathcal{R}_{i}^{\eta}$. The points $\left\{h_{i}^{\eta}\right\}_{i=1,\ldots,L}$ should be seen as a discretization of the space $\mathcal{X}$. By virtue of \cite[Proposition 7.9]{Folland2013} and \cite[Lemma 19.24]{van_der_Vaart2000}, if the class $\mathcal{C}_{\operatorname{b}}^{s}(\mathcal{X})_{1}$ is $\eta$-Donsker and $\left\{U_{i,l}^{\eta}\right\}_{i=1,\ldots,L}$ are analogous to \eqref{Eqn:NaiveU}, then $\sqrt{l}\,\left(\widehat{\eta}_{L,l}-\widehat{\eta}_{L}\right)\rightsquigarrow\mathbb{G}^{\eta}$ when $l\longrightarrow\infty$ on $\mathcal{F}_{\eta}^{\operatorname{pixel}}$. Equivalently, by Cr\'{a}mer-Wold's Theorem (\cite[p. 16]{van_der_Vaart2000}) and Portmanteau's Theorem (\cite[Theorem 2.1]{Billingsley2013}),
		\begin{equation}\label{Eqn:ConvergenceU}
			\sqrt{l}\,\left(\mathbb{U}_{l}^{\eta}-\mathbf{m}^{\eta}\right)\rightsquigarrow\mathcal{N}\left(0,\Sigma_{\eta}\right),
		\end{equation}
		when $l\longrightarrow\infty$, where $\Sigma_{\eta}=\sum_{i=1}^{L}\eta\left(\mathcal{R}_{i}^{\eta}\right)\,e_{i}\,e_{i}^{\top}-\mathbf{m}^{\eta}\,\left(\mathbf{m}^{\eta}\right)^{\top}$, $\left\{e_{i}\right\}_{i=1,\ldots,L}$ are the (column) vectors of the canonical basis and $^{\top}$ means transposition of matrices.
		
		The goal of this subsection is providing consistency results for the methodology proposed in \cite{Sommerfeld2019} under the framework of this paper, that is, for the supremum norm, given \eqref{Eqn:ConvergenceU}. Since the proof of main results is strongly based on the application of delta method, it is enough to prove the convergence of the underlying process.
		
		Define the (resampling) measure as
		\begin{equation}
			\widehat{\eta}_{L,l}^{\ast}=\sum_{i=1}^{L}M_{i,l}^{\eta}\,\delta_{h_{i}^{\eta}},
		\end{equation}
		where $l\,\mathbb{M}_{l}^{\eta}=\left(M_{i,l}^{\eta}\right)_{i=1,\ldots,L}^{\top}\sim\operatorname{Mult}\left(l;U_{1,l}^{\eta},\ldots,U_{L,l}^{\eta}\right)$, a multinomial distribution.
		\begin{Lema}
			Given $\eta\in\mathcal{P}(\mathcal{X})$ and $\mathbf{m}^{\eta}$ and $\mathbb{U}_{l}^{\eta}$ as above, assume \eqref{Eqn:ConvergenceU}. Then,
			\begin{equation}\label{Eqn:ConvergenceBootsU}
				\sqrt{l}\,\left(\mathbb{M}_{l}^{\eta}-\mathbb{U}_{l}^{\eta}\right)\rightsquigarrow\mathcal{N}\left(0,\Sigma_{\eta}\right)
			\end{equation}
			when $l\longrightarrow\infty$. Equivalently, in terms of $\widehat{\eta}_{L,l}$ and $\widehat{\eta}_{L,l}^{\ast}$: $\sqrt{l}\,\left(\widehat{\eta}_{L,l}^{\ast}-\widehat{\eta}_{L,l}\right)\rightsquigarrow\mathbb{G}^{\eta}$ on $\mathcal{F}_{\eta}^{\operatorname{pixel}}$.
		\end{Lema}
		\begin{proof}
			As in the previous paragraphs, the equivalence between both limits is given by Cr\'{a}mer-Wold's Theorem and Portmanteau's Theorem. Hence, it is enough to prove \eqref{Eqn:ConvergenceBootsU}. Since it is a multivariate random variable, characteristic functions are used. Hence, for $\mathbf{t}\in\mathbb{R}^{L}$:
			\begin{equation}
				\begin{aligned}
					f_{l}(\mathbf{t})&=\mathbb{E}\left(\left.\operatorname{exp}\left(\iota\,\sqrt{l}\,\left(\mathbb{M}_{l}^{\eta}-\mathbb{U}_{l}^{\eta}\right)^{\top}\,\mathbf{t}\right)\right|\mathbb{U}_{l}^{\eta}\right)
					\\
					&=\left(\sum_{i=1}^{L}U_{i,l}^{\eta}\,\operatorname{exp}\left(\frac{\iota}{\sqrt{l}}\,\left(\sum_{i=1}^{L}U_{i,l}^{\eta}\,e_{i}\,e_{i}^{\top}-\mathbb{U}_{l}^{\eta}\,\left(\mathbb{U}_{l}^{\eta}\right)^{\top}\right)\right)\right)^{l}.
				\end{aligned}
			\end{equation}
			Now, use the Taylor's expansion of second order of the exponential function and the fact that $\sum_{i=1}^{L}U_{i,l}^{\eta}=1$ to conclude that
			\begin{equation}
				f_{l}(\mathbf{t})=\left(1-\frac{1}{2\,l}\,\mathbf{t}^{\top}\,\left(\sum_{i=1}^{L}U_{i,l}^{\eta}\,e_{i}\,e_{i}^{\top}-\mathbb{U}_{l}^{\eta}\,\left(\mathbb{U}_{l}^{\eta}\right)^{\top}\right)\,\mathbf{t}+o\left(\frac{\|\mathbf{t}\|_{2}^{2}}{l}\right)\right)^{l}.
			\end{equation}
			Observe that \eqref{Eqn:ConvergenceU} implies that for all $i\in\{1,\ldots,L\},\ U_{i,l}^{\eta}\overset{\operatorname{Prob}}{\longrightarrow}\eta\left(\mathcal{R}_{i}^{\eta}\right)$ when $l\longrightarrow\infty$. So, $f_{l}(\mathbf{t})\longrightarrow\operatorname{exp}\left(-\frac{\mathbf{t}^{\top}\,\Sigma_{\eta}\,\mathbf{t}}{2}\right)$ when $l\longrightarrow\infty$. By \cite[Theorem 1.13.1 (iv)]{van_der_Vaart&Wellner2023} the proof is ended.
		\end{proof}
		It is worth to mention that to obtain the limit \eqref{Eqn:ConvergenceBootsU}, it is only needed that $i\in\{1,\ldots,L\},\ U_{i,l}^{\eta}\overset{\operatorname{Prob}}{\longrightarrow}\eta\left(\mathcal{R}_{i}^{\eta}\right)$ when $l\longrightarrow\infty$. However, there is no warranty of \eqref{Eqn:ConvergenceU} and the limits might differ (if they exist). Consequently, \eqref{Eqn:ConvergenceBootsU} might not be interesting for applications. Usual frameworks where \eqref{Eqn:ConvergenceU} is given include the scenario provided in equation \eqref{Eqn:NaiveU} and small perturbations of it (in mathematical terms: $\sqrt{l}\,\left\|\widehat{\eta}_{L,l}-\eta_{l}\right\|_{\ell^{\infty}\left(\mathcal{F}_{\eta}^{\operatorname{pixel}}\right)}\overset{\operatorname{Prob}}{\longrightarrow}0$ when $l\longrightarrow\infty$).
		
\end{document}